\documentclass[10pt]{amsart}
\usepackage{enumerate,  mathrsfs,color,amssymb}

\usepackage[utf8]{inputenc}
\usepackage[T1]{fontenc}

\usepackage{amsmath}
\usepackage{float}
\usepackage[all]{xy}
\usepackage{fancyhdr}
\usepackage{graphicx}

\usepackage{mathtools}
\usepackage{adjustbox}
\usepackage{enumitem}
\usepackage{stackrel}
\usepackage{scalerel}
\usepackage{euscript}
\usepackage{amsmath}
\usepackage{amsthm}
\usepackage{indentfirst}

\usepackage{mathpazo}
\usepackage{mathrsfs}
\usepackage{amsmath}
\usepackage{amssymb}
\usepackage{hyperref}
\usepackage{cleveref}
\usepackage{comment}
\usepackage{color}
\usepackage{tikz}
\usepackage{tikz-cd}
\usepackage{xparse}
\usetikzlibrary{matrix,arrows,decorations.pathmorphing}

\newtheorem{theorem}{Theorem}
\newtheorem{proposition}{Proposition}
\newtheorem{corollary}{Corollary}
\newtheorem{lemma}{Lemma}

\newtheorem{remark}{Remark}

\usepackage{amsfonts} 
\DeclareMathSymbol{\shortminus}{\mathbin}{AMSa}{"39}
\usepackage{mathpazo}

\newcommand{\titleinfo}{On some Lie automorphisms of a class of Kadison-Singer algebras} 
\newcommand{\titleinfoshort}{Lie automorphisms}
\newcommand{\authorinfo}{}

\begin{document}

\title{\LARGE\textbf{\titleinfo}}





\author[Yang]{ Zhujun Yang}
\address{School of Mathematical Sciences, Qufu Normal University, Qufu, Shandong, 273165, China}
\email{zjyang2100@163.com}

\author[Zhou]{ You Zhou}
\address{School of Mathematical Sciences, Qufu Normal University, Qufu, Shandong, 273165, China}
\email{yzhou201302@163.com}

\author[Wang]{  Liguang Wang}
\address{School of Mathematical Sciences, Qufu normal University, Qufu 273165, China.}
\email{wangliguang0510@163.com}

\begin{abstract}
Let $\mathcal{H}$ be an infinite dimensional separable Hilbert space and $\mathcal{N}$ a nest of projections on $\mathcal{H}$ with at least four projections. Let $\xi$ be a separating vector of $\mathcal{N}^{''}$ and $P_{\xi}$ the orthogonal projection from $\mathcal{H}$ onto the one-dimensional subspace of $\mathcal{H}$ generated by $\xi$. Let $\mathcal{L}$ be the lattice generated by $\mathcal{N}$ and $P_{\xi}$, and ${\rm{Alg}}\mathcal{L}$ the corresponding Kadison-Singer algebra. In this note, we show that every Lie automorphism $\psi$ on ${\rm{Alg}}\mathcal{L}$ can be decomposed as $\psi=\epsilon+\tau$ when $I_{-}^{\mathcal{N}}\vee P_{\xi}<I$, where $\epsilon$ is an automorphism and $\tau$ is a linear functional on ${\rm{Alg}}\mathcal{L}$ vanishing on each commutator. When  $I_{-}^{\mathcal{N}}<I$ with $I_{-}^{\mathcal{N}}\vee P_{\xi}=I$, let $\mathcal{N}_{1}=\{N\in\mathcal{N}:N\leq I_{-}^{\mathcal{N}}\}$ and $\mathcal{T}(\mathcal{N}_{1})$ be the corresponding nest algebra. Then ${\rm{Alg}}\mathcal{L}$ is isomorphic to $\mathcal{T}(\mathcal{N}_{1})\oplus \mathbb{C}$. We show that for every Lie automorphism $\psi:\mathcal{T}(\mathcal{N}_{1})\oplus \mathbb{C}\rightarrow \mathcal{T}(\mathcal{N}_{1})\oplus \mathbb{C}$, there exist an (algebraic) isomorphism or a negative of an (algebraic) anti-isomorphism $\epsilon$ on $\mathcal{T}(\mathcal{N}_{1})$, a linear functional $\tau_{2}:\mathbb{C}\to\mathbb{C}$ and two linear functionals $\tau_{1},\tau$ vanishing on the commutators of $\mathcal{T}(\mathcal{N}_{1})$ and $\mathcal{T}(\mathcal{N}_{1})\oplus \mathbb{C}$ respectively, such that
$\psi=((\epsilon+\tau_{1})\oplus \tau_{2})+\tau.$
\end{abstract}


\subjclass[2010]{47L75,46L10}
\keywords{Kadison-Singer algebras; Lie ideals; Lie automorphisms.}
\date{}
\maketitle

\section{Introduction}

Let $\mathcal{A}$ be an associative algebra. A Lie homomorphism $\psi$ of $\mathcal{A}$ into another associative algebra is a linear map which preserves the Lie product. That is, $\psi([A,B])=[\psi(A),\psi(B)]$ for all $A,B\in\mathcal{A}$.
As usual, a bijective Lie homomorphism is called a Lie isomorphism, and if a Lie isomorphism $\psi$ maps $\mathcal{A}$ onto itself, it is termed a Lie automorphism.

The study of Lie isomorphisms of associative and operator algebras—particularly their connection with associative (anti-)isomorphisms-has a long history; see \cite{BBCM,Bresar} and the references therein. A Lie isomorphism
$\psi:\mathcal{A}\to\mathcal{B}$ is said to be standard if it can be decomposed as $\psi=\epsilon+\tau$ , where $\epsilon$ is an (algebraic) isomorphism or a negative of an (algebraic) anti-isomorphism, and $\tau$ is a linear map with image in the center vanishing on each commutator.

In \cite{Miers}, Miers proved that every Lie isomorphism of (von Neumann) factors is standard. Later, Marcoux and Sourour \cite{Marcoux-Sourour} extended this result to nest algebras. For characterizations of Lie isomorphisms on other non-self-adjoint operator algebras, we refer to \cite{Lu2,Lu4,Martindale}.

Kadison-Singer algebras are a new type of non-self-adjoint operator algebras introduced by Ge and Yuan in \cite{Ge-Yuan1,Ge-Yuan2}. Kadison-Singer algebras are reflexive and maximal with respect to their diagonals. The corresponding reflexive lattice is called Kadison-Singer lattice. Wang and Yuan (\cite{Wang-Yuan}) gave a new class of Kadison-Singer algebras with  lattices  generated by a maximal nest and a projection. Hou (\cite{Hou}) generalized results in \cite{Wang-Yuan} and  constructed a new class of Kadison-Singer algebras. Recently, more examples of Kadison-Singer algebras have been introduced and analyzed in \cite{Wu}.

The Kadison–Singer algebras constructed above are closely related to nest algebras. Naturally, we begin to consider Lie isomorphisms on these algebras. In \cite{YZW}, the properties of Lie automorphisms on a class of Kadison–Singer algebras introduced in \cite{Wang-Yuan} are characterized. The algebras constructed in \cite{Hou} are a generalization of those in \cite{Wang-Yuan}. Therefore, we naturally ask what structure Lie isomorphisms on these more general algebras may have.



Now we introduce some classes
 of the Kadison-Singer lattices that we will consider in this paper. The following notations will be fixed throughout this paper.

Let $\mathcal{H}$ be an infinite dimensional separable Hilbert space and $\mathcal{N}$ a nest of projections on $\mathcal{H}$ with at least four projections. The corresponding nest algebra is denoted by $\mathcal{T}(\mathcal{N})$. Let $\xi$ be a separating vector of $\mathcal{N}^{''}$ and $P_{\xi}$ the orthogonal projection from $\mathcal{H}$ onto the one-dimensional subspace of $\mathcal{H}$ generated by $\xi$. Then it follows from \cite{Hou} that the subspace lattice
$$
\mathcal{L}=\{0,I,P,P_{\xi},P\vee P_{\xi}:P\in\mathcal{N}\}
$$
is a Kadison-Singer lattice.

For different nests, since the predecessor of the identity in the Kadison-Singer lattice $\mathcal{L}$ admits three possibilities, we also consider three cases when discussing Lie automorphisms on the corresponding Kadison-Singer algebras: (1) $I_{-}^{\mathcal{N}}\neq I$ with $I_{-}^{\mathcal{N}}\vee P_{\xi}<I$; (2) $I_{-}^{\mathcal{N}}\neq I$ with $I_{-}^{\mathcal{N}}\vee P_{\xi}=I$; and (3) $I_{-}^{\mathcal{N}}=I$. The proof for the case $I_{-}^{\mathcal{N}}=I$ coincides with that for Kadison–Singer algebras associated with maximal nests (see \cite{YZW}); therefore, in this paper we only treat the first two cases.

\section{Preliminaries}
In the following sections, we always assume $I_{-}^{\mathcal{N}}\vee P_{\xi}<I$; the case $I_{-}^{\mathcal{N}}<I$ with $I_{-}^{\mathcal{N}}\vee P_{\xi}=I$ will be considered in the last section.

By calculation, we have $\wedge\{Q_{-}:Q\in\mathcal{N}\}=P_{\xi}$. And for convenience, we let $\mathcal{J}(\mathcal{L})=\{P\in\mathcal{L}:P\neq0,P_{-}\neq I\}$ and $\mathcal{T}(\mathcal{N})|_{N}=\{T|_{N}:T\in{\rm{Alg}}\mathcal{L}\}$.

\begin{proposition}\label{idem}
Let $\mathcal{L}$ be the  subspace lattice given in section 1 and $A\in{\rm{Alg}}\mathcal{L}$. Then
\begin{itemize}
  \item [(1)] $A$ is the sum of a scalar and an idempotent if and only if
  $$
  [A,[A,[A,T]]]=[A,T] \ {\rm{for \ every}} \ T\in{\rm{Alg}}\mathcal{L}.
  $$
  \item [(2)] $A$ is the sum of a scalar and an idempotent whose range belongs to $\mathcal{L}$ if and only if
  $$
  [A,[A,T]]=[A,T] \ {\rm{for \ every}} \ T\in{\rm{Alg}}\mathcal{L}.
  $$
\end{itemize}
\end{proposition}

\begin{proof}
(1) The sufficiency is obvious, so we only prove the necessity. Assume now that
$$
  [A,[A,[A,T]]]=[A,T] \ {\rm{for \ every}} \ T\in{\rm{Alg}}\mathcal{L},
$$
then
\begin{equation}\label{1}
  (A^{3}-A)T-3A^{2}TA+3ATA^{2}-T(A^{3}-A)=0 \ {\rm{for \ every}} \ T\in{\rm{Alg}}\mathcal{L}.
\end{equation}

Since $I_{-}^{\mathcal{N}}\vee P_{\xi}<I$, we have $I\in\mathcal{J}(\mathcal{L})$.

Let $P\in\mathcal{J}(\mathcal{L})$. Then $x\otimes f\in{\rm{Alg}}\mathcal{L}$ for all $x\in P$ and $f\in P_{-}^{\perp}$. Upon taking $T=x\otimes f$ in Eq.(\ref{1}) and then applying this equation to a vector $z\in\mathcal{H}$, we have
\begin{equation}\label{1'}
  \langle z,f\rangle(A^{3}-A)x-3\langle Az,f\rangle A^{2}x+3\langle A^{2}z,f\rangle Ax-\langle (A^{3}-A)z,f\rangle x=0.
\end{equation}
We now consider two cases:

Case a. There is some $P\in\mathcal{J}(\mathcal{L})$ such that $A^{\ast}f$ and $f$ are linearly dependent for any $f\in P_{-}^{\perp}$. Then there is some scalar $\lambda_{P}\in\mathbb{C}$ such that $A^{\ast}f=\lambda_{P}f$ for all $f\in P_{-}^{\perp}$. 
Without loss of generality, we can replace $A-\bar{\lambda}_{P} I$ with $A$, then we have $A^{\ast}f=0$ for all $f\in P_{-}^{\perp}$ and $(I-P_{-})A=0$. Since $P\leq I$, then $P_{-}\leq I_{-}$ and $P_{-}^{\perp}\geq I_{-}^{\perp}$. Thus $A^{\ast}f=0$ for all $f\in I_{-}^{\perp}$, by Eq.(\ref{1'}), we have $(A^{3}-A)x=0$ for all $x\in\mathcal{H}$ and $A^{3}=A$. 
On the other hand, from Eq. (\ref{1}) we obtain the identity
\begin{equation}\label{A^3=A}
 ATA^{2}=A^{2}TA, \ {\rm{for \ every}} \ T\in{\rm{Alg}}\mathcal{L}.
\end{equation}
Using the decomposition $\mathcal{H}=(I_{-}^{\mathcal{N}})\oplus (I_{-}^{\mathcal{N}})^{\perp}$ and writing $\xi=\xi_{1}\oplus \xi_{2}$, we represent the operators as
$$
A=\left(\begin{array}{cc}
A_{11}&A_{12}\\
0&A_{22}
\end{array}\right) \ {\rm{and}} \
T=\left(\begin{array}{cc}
T_{11}&T_{12}\\
0&T_{22}
\end{array}\right).
$$

Since $A^{\ast}f=0$ for all $f\in I_{-}^{\perp}$, we have $A_{22}=\xi_{2}\otimes f$ for some $f_{0}\in(I_{-}^{\mathcal{N}})^{\perp}$. From $A^{3}=A$, it follows that $A_{22}=0$, $A_{22}$ is idempotent when $\langle \xi_{2},f_{0}\rangle=1$, and $A_{22}$ is a negative idempotent when $\langle \xi_{2},f_{0}\rangle=-1$.

By \cite{Marcoux-Sourour}, there is an idempotent $X$ such that $A_{11}=X+\lambda (I_{-}^{\mathcal{N}})^{\perp}$. Since $A_{11}^{3}=A_{11}$, we obtain either $A_{11}=X$ or $A_{11}=X-1$ with $X$ an idempotent.

If $A_{11},A_{22}$ are idempotents, substituting the block forms into Eq.(\ref{A^3=A}) yields
\begin{align*}
A_{11}T_{11}A_{11}A_{12}+A_{11}T_{11}A_{12}A_{22}-A_{11}T_{11}A_{12} \\
=A_{11}A_{12}T_{22}A_{22}+A_{12}A_{22}T_{22}A_{22}-A_{12}T_{22}A_{22}.
\end{align*}
Now, take $T_{11}=I_{-}^{\mathcal{N}}$, $T_{22}=0$ and $T_{12}=\frac{-\xi_{1}\otimes \xi_{2}}{\|\xi_{2}\|^{2}}$ in the above equation. This implies $A_{11}A_{12}A_{22}=0$. Together with the condition $A^{3}=A$ and the fact that $A_{11},A_{22}$ are idempotent, which implies $A_{11}A_{12}+A_{12}A_{22}=A_{12}$, we conclude that $A^{2}=A$.

If $A_{11}=X-1$ with $X$ and $A_{22}=\xi_{2}\otimes f_{0}$ are idempotents, we first assume $A_{22}=0$. From $A^{3}=A$, we have $XA_{12}=0$, and hence $A+1$ is an idempotent. Now, we assume that $A_{22}=\xi_{2}\otimes f_{0}$ with $\langle \xi_{2},f_{0}\rangle=1$. Then $A^{3}=A$ gives
$$
XA_{12}=XA_{12}A_{22}.
$$
Substituting this into Eq.(\ref{A^3=A}) and together with the fact that $A,T\in{\rm{Alg}\mathcal{L}}$, we deduce $X=1$ and $A_{12}=A_{12}A_{22}$. Consequently, $A$ itself is an idempotent.

If $A_{11}$ is an idempotent and $A_{22}=\xi_{2}\otimes f_{0}$ with $\langle \xi_{2},f_{0}\rangle=-1$, then form $A^{3}=A$, we obtain
$$
A_{11}A_{12}-A_{12}=(A_{11}-1)\xi_{1}\otimes f_{0}.
$$
Substituting this into Eq.(\ref{A^3=A}) yields $A_{11}=0$, and consequently $A_{12}=\xi_{1}\otimes f_{0}$. Thus $A_{12}(A_{22}+1)=0$. Therefore, $A+1$ is idempotent.

If $A_{11}=X-1$ and $A_{22}=\xi_{2}\otimes f_{0}$ with $X$ is an idempotent and $\langle \xi_{2},f_{0}\rangle=-1$. A direct calculation yields $A_{12}A_{22}=-X\xi_{1}\otimes f_{0}$. From $A^{3}=A$, we have $XA_{12}=X\xi_{1}\otimes f_{0}$. It is then clear that $XA_{12}+A_{12}A_{22}=0$, and consequently $A+1$ is idempotent.

Case b. For any $P\in\mathcal{J}(\mathcal{L})$, there is some $f\in P_{-}^{\perp}$ such that $A^{\ast}f$ and $f$ are linearly independent. Since $I\in\mathcal{J}(\mathcal{L})$. Then there is some $f\in I_{-}^{\perp}$ such that $A^{\ast}f$ and $f$ are linearly independent, then we can choose some vector $z\in\mathcal{H}$ such that $\langle z,f\rangle=0$ and $\langle z,A^{\ast}f\rangle=1$. It follows from Eq.(\ref{1'}) that there is a monic quadratic polynomial $q$ such that $q(A)=0$. Now translate $A$ so that $A^{2}=cI$ for some scalar $c\in\mathbb{C}$, by Eq.(\ref{1}), then we have
$$
(4c-1)AT=(4c-1)TA, \ \forall \ T\in{\rm{Alg}}\mathcal{L}.
$$
If $4c-1=0$, then $A^{2}=\frac{1}{4}I$ and $A+\frac{1}{2}I$ is an idempotent. If $4c-1\neq 0$, then
$$
AT=TA, \ \forall \ T\in{\rm{Alg}}\mathcal{L}.
$$
Thus $A$ is an operator in center of ${\rm{Alg}}\mathcal{L}$, and $A$ is a scalar.

(2) Let $P\in{\rm{Alg}}\mathcal{L}$ be an idempotent operator whose range is in $\mathcal{L}$. Since the range of $P$ is a closed subspace (hence a Hilbert space), by the decomposition $\mathcal{H}={\rm{Ran}}(P)\oplus {\rm{Ran}}(P)^{\perp}$, we have
$$
P=\left(\begin{array}{cc}
1&P_{12}\\
0&0
\end{array}\right).
$$
Let $A$ be the sum of a scalar and an idempotent whose range is in $\mathcal{L}$. Then, by a straightforward calculation, we obtain $[A,[A,T]]=[A,T]$ for any $T\in{\rm{Alg}}\mathcal{L}$.

Now, we assume that $[A,[A,T]]=[A,T]$ for any $T\in{\rm{Alg}}\mathcal{L}$. Then there are a scalar $\lambda\in\mathbb{C}$ and an idempotent $E\in{\rm{Alg}}\mathcal{L}$ such that $A=\lambda I+E$. By direct calculation, we have $[E,[E,T]]=[E,T]$ for all $T\in{\rm{Alg}}\mathcal{L}$, which implies $TE=ETE$ for all $T\in{\rm{Alg}}\mathcal{L}$. Thus ${\rm{Ran}}(E)\in{\rm{LatAlg}}\mathcal{L}=\mathcal{L}$, since $\mathcal{L}$ is a reflexive lattice.

\end{proof}
\begin{remark}\label{CLI}
Define
$$
\mathcal{J}(\mathcal{L},E)=\{A\in{\rm{Alg}}\mathcal{L}:AE=0 \ and \ A^{\ast}E^{\perp}=0\}.
$$
It is not difficult to verify that $\mathcal{J}(\mathcal{L},E)$ is an ideal and $\mathcal{J}(\mathcal{L},0)$, $\mathcal{J}(\mathcal{L},I)$ are trivial. For any $A,B\in\mathcal{J}(\mathcal{L},E)$, we have $AB=0$, and hence $\mathcal{J}(\mathcal{L},E)$ and $\mathbb{C}I+\mathcal{J}(\mathcal{L},E)$ are commutative Lie ideals.

For any $A\in\mathcal{J}(\mathcal{L},N)$ with $N\neq I$, we have $AN=0$ and $N^{\perp}A=0$. It is clear that $(N\vee P_{\xi})^{\perp}A=0$. Since $N\neq I$, then $A\xi=0$ and $A(N\vee P_{\xi})=0$. Thus $\mathcal{J}(\mathcal{L},N)\subset \mathcal{J}(\mathcal{L},N\vee P_{\xi})$.

For any $N\leq I_{-}^{\mathcal{N}}$, if we take $f\in(I_{-}^{\mathcal{N}}\vee P_{\xi})^{\perp}$, then by direct computation we have $\xi\otimes f\in\mathcal{J}(\mathcal{L},N\vee P_{\xi})$ but $\xi\otimes f\notin\mathcal{J}(\mathcal{L},N)$. In fact, we can prove that
$$
\mathcal{J}(\mathcal{L},N\vee P_{\xi})=\mathcal{J}(\mathcal{L},N)+\{\xi\otimes f:f\in(I_{-}^{\mathcal{N}}\vee P_{\xi})^{\perp}\}.
$$
For any $T\in \mathcal{J}(\mathcal{L},N\vee P_{\xi})\setminus \mathcal{J}(\mathcal{L},N)$, since $A(I_{-}^{\mathcal{N}}\ominus N)\subset I_{-}^{\mathcal{N}}\wedge (N\vee P_{\xi})=N$, then by the decomposition $\mathcal{H}=N\oplus (I_{-}^{\mathcal{N}}\ominus N)\oplus (I_{-}^{\mathcal{N}})^{\perp}$, we can assume
$$
T=\left(\begin{array}{ccc}
0&T_{12}&T_{13}\\
0&0&T_{23}\\
0&0&T_{33}
\end{array}\right),
$$
where $T_{23}\neq0$ or $T_{33}\neq0$. For convenience, let $\xi=(\xi_{1},\xi_{2},\xi_{3})^{T}$. Since $T(I_{-}^{\mathcal{N}})^{\perp}\subset N\vee P_{\xi}$ and $T\xi=0$, there exists a nonzero vector $f\in(I_{-}^{\mathcal{N}}\vee P_{\xi})^{\perp}$ such that
$$
\left(\begin{array}{cc}
0&T_{23}\\
0&T_{33}
\end{array}\right)=
\left(\begin{array}{c}
\xi_{2}\\
\xi_{3}
\end{array}\right)\otimes f.
$$
Thus $T-\xi\otimes f\in\mathcal{J}(\mathcal{L},N)$, and hence
$$
\mathcal{J}(\mathcal{L},N\vee P_{\xi})=\mathcal{J}(\mathcal{L},N)+\{\xi\otimes f:f\in(I_{-}^{\mathcal{N}}\vee P_{\xi})^{\perp}\}.
$$

In particular, if $N=0$, then $\mathcal{J}(\mathcal{L},P_{\xi})=\{\xi\otimes f:f\in(I_{-}^{\mathcal{N}}\vee P_{\xi})^{\perp}\}$. Indeed, for any $A\in\mathcal{J}(\mathcal{L},P_{\xi})$, we have $A\mathcal{H}\subset \mathbf{C}\xi$, hence, $A$ is a rank-one operator. Since $A\xi=0$, then there exists $f\in(I_{-}^{\mathcal{N}}\vee P_{\xi})^{\perp}$ such that $A=\xi\otimes f$. The above discussion implies that $\mathcal{J}(\mathcal{L},P_{\xi})\subset \mathcal{J}(\mathcal{L},N\vee P_{\xi})$ for all $N\in\mathcal{N}^{0}$.

\end{remark}

\begin{proposition}\label{ATA=0}
Let $A\in{\rm{Alg}}\mathcal{L}$ be a nonzero element. If $ATA=0$ for any $T\in{\rm{Alg}}\mathcal{L}$, then there is some $P\in\mathcal{L}$ such that $AP=0$ and $A^{\ast}P^{\perp}=0$.
\end{proposition}
\begin{proof}
Let $P_{0}=\vee\{P\in\mathcal{L}:AP=0\}$. Then $AP_{0}=0$. Since $A\neq 0$, we have $P_{0}\neq I$. It is clear that $A^{2}=0$ and $A\xi=0$, then $P_{\xi}\leq P_{0}\neq 0$ and there is $N(\neq I)\in\mathcal{N}$ such that $P_{0}=N\vee P_{\xi}$. Next, we claim that $A^{\ast}P_{0}^{\perp}=0$.

Firstly, we let $N_{+}^{\mathcal{N}}\neq N$. Then $N_{+}^{\mathcal{N}}=I$ or $P_{0}$ and $N_{+}^{\mathcal{N}}$ cannot be compared. If $N_{+}^{\mathcal{N}}=I$, then $N=I_{-}^{\mathcal{N}}$ and $P_{0}=I_{-}^{\mathcal{N}}\vee P_{\xi}<I$. Thus $A(I_{-}^{\mathcal{N}}\vee P_{\xi})=0$, since $ATA=0$ for any $T\in{\rm{Alg}}\mathcal{L}$ and $\mathcal{B}((I-I_{-}^{\mathcal{N}}\vee P_{\xi})\mathcal{H})$ is semi-prime, we have $(I-I_{-}^{\mathcal{N}}\vee P_{\xi})A(I-I_{-}^{\mathcal{N}}\vee P_{\xi})=0$, hence, we have $A^{\ast}P_{0}^{\perp}=0$. If $P_{0}$ and $N_{+}^{\mathcal{N}}$ cannot be compared, then there is some nonzero element $x\in N_{+}^{\mathcal{N}}$ such that $Ax\neq 0$, thus $A^{\ast}f=0$ for all $f\in (N_{+}^{\mathcal{N}})_{-}^{\perp}$, by calculation, we have $(N_{+}^{\mathcal{N}})_{-}=N\vee P_{\xi}=P_{0}$, thus $A^{\ast}P_{0}^{\perp}=0$.

Secondly, we let $N_{+}^{\mathcal{N}}=N$. Then there are $N<N_{i}\in\mathcal{N}$ such that $N_{i}>N_{i+1}$ and $\wedge\{N_{i}\}_{i\in\mathbf{N}}=N$. Similarly, we have $A^{\ast}(N_{i})_{-}^{\perp}=0$, and by calculation, we have $\vee[(N_{i})_{-}^{\perp}]=(\wedge(N_{i})_{-})^{\perp}=[\wedge(N_{i})_{-}^{\mathcal{N}}\vee P_{\xi}]^{\perp}=(N\vee P_{\xi})^{\perp}=P_{0}^{\perp}$. Thus $A^{\ast}P_{0}^{\perp}=0$.

\end{proof}
\begin{proposition}\label{[A,[A,T]]=0}
Let $A\in{\rm{Alg}}\mathcal{L}$. Then $A\in\mathbb{C}I+\mathcal{J}(\mathcal{L},E)$ for some $E\in\mathcal{L}$ if and only if
$$
  [A,[A,T]]=0 \ {\rm{for \ every}} \ T\in{\rm{Alg}}\mathcal{L}.
$$

\end{proposition}

\begin{proof}
(1) The sufficiency is obvious, so we only prove the necessity. Assume now that
$$
  [A,[A,T]]=0 \ {\rm{for \ every}} \ T\in{\rm{Alg}}\mathcal{L},
$$
then for any $T\in{\rm{Alg}}\mathcal{L}$, we have
\begin{equation}\label{3.1}
 A^{2}T-2ATA+TA^{2}=0.
\end{equation}

Let $P\in\mathcal{J}(\mathcal{L})$. Then $x\otimes f\in{\rm{Alg}}\mathcal{L}$ for all $x\in P$ and $f\in P_{-}^{\perp}$. Upon taking $T=x\otimes f$ in Eq.(\ref{3.1}) and then applying this equation to a vector $z\in\mathcal{H}$, we have
\begin{equation}\label{3.1'}
  \langle z,f\rangle A^{2}x-2\langle Az,f\rangle Ax+\langle A^{2}z,f\rangle x=0.
\end{equation}
We now consider two cases:

Case a. There is some $P\in\mathcal{J}(\mathcal{L})$ such that $A^{\ast}f$ and $f$ are linearly dependent for any $f\in P_{-}^{\perp}$. 
Then there is some $\lambda_{P}\in\mathbf{C}$ such that $A^{\ast}f=\lambda_{P} f$ for any vector $f\in P_{-}^{\perp}$. Without loss of generality, we can replace $A-\bar{\lambda}_{P} I$ with $A$, then we have $A^{\ast}f=0$ for all $f\in P_{-}^{\perp}$ and $(I-P_{-})A=0$. Since $P\leq I$, then $P_{-}\leq I_{-}$ and $P_{-}^{\perp}\geq I_{-}^{\perp}$. Thus $A^{\ast}f=0$ for all $f\in I_{-}^{\perp}$, by Eq.(\ref{3.1'}), we have $A^{2}x=0$ for all $x\in\mathcal{H}$ and $A^{2}=0$. Then by Eq.(\ref{3.1}), we have $ATA=0$ for any $T\in{\rm{Alg}}\mathcal{L}$. Hence $A\in\mathbb{C}I+\mathcal{J}(\mathcal{L},E)$ for some $E\in\mathcal{L}$ by Proposition \ref{ATA=0}.

Case b. For any $P\in\mathcal{J}(\mathcal{L})$, there is some $f\in P_{-}^{\perp}$ such that $A^{\ast}f$ and $f$ are linearly independent. Since $I\in\mathcal{J}(\mathcal{L})$. Then there is some $f\in I_{-}^{\perp}$ such that $A^{\ast}f$ and $f$ are linearly independent, then we can choose some vector $z\in\mathcal{H}$ such that $\langle z,f\rangle=0$ and $\langle z,A^{\ast}f\rangle=1$. It follows from Eq.(\ref{3.1'}) that there is some $\lambda\in\mathbb{C}$ such that $Ax=\lambda x$ for all $x\in\mathcal{H}$, hence $A=\lambda I$ and the proof is complete.

\end{proof}

\begin{theorem}\label{J(L,E)}
Let $\mathcal{L}$ be the  subspace lattice given in section 1. Then the following hold:
\begin{itemize}
  \item [(1)] Let $N\in\mathcal{N}^{0}$ and $E=N\vee P_{\xi}$. Then $\mathbb{C}I+\mathcal{J}(\mathcal{L},E)$ is a maximal commutative Lie ideal in ${\rm{Alg}}\mathcal{L}$.
  \item [(2)] For any non-trivial commutative Lie ideal $\mathcal{J}_{0} \subset {\rm{Alg}}\mathcal{L}$, there exists $N \in \mathcal{N}^{0}$ such that $\mathcal{J}_{0} \subset \mathbb{C}I + \mathcal{J}(\mathcal{L},N\vee P_{\xi})$. Moreover, if $\mathcal{J}_{0}$ is maximal, then equality holds and $N$ is uniquely determined.
\end{itemize}
\end{theorem}
\begin{proof}
(1) Let $\mathcal{R}$ be a commutative Lie ideal in ${\rm{Alg}}\mathcal{L}$ which contains $\mathbb{C}I+\mathcal{J}(\mathcal{L},E)$ and $A\in\mathcal{R}$. Next we prove that $\mathcal{R}\subset \mathbb{C}I+\mathcal{J}(\mathcal{L},E)$.

For all $x\in N \subset N\vee P_{\xi}$ and $f\in (N \vee P_{\xi})^{\perp}$, we have $x\otimes f\in \mathcal{J}(\mathcal{L},N) \subset\mathcal{J}(\mathcal{L},E)$, then $A(x\otimes f)=(x\otimes f)A$. Thus there is some scalar $\lambda\in\mathbb{C}$ such that $Ax=\lambda x$ and $A^{\ast}f=\bar{\lambda}f$ for all $x\in N$ and $f\in (N \vee P_{\xi})^{\perp}$. Similarly, for any $y\in N\vee P_{\xi}$ and $g\in (I_{-}^{\mathcal{N}} \vee P_{\xi})^{\perp}$, we have $A(y\otimes g)=(y\otimes g)A$, then there is some scalar $\mu\in\mathbb{C}$ such that $Ay=\mu y$ and $A^{\ast}g=\bar{\mu}g$ for all $y\in N\vee P_{\xi}$ and $g\in (I_{-}^{\mathcal{N}} \vee P_{\xi})^{\perp}$. Since $0\neq N\subset N\vee P_{\xi}$, then $\lambda=\mu$. Considering $A-\lambda I$, then $(A-\lambda I)y=0$ and $(A-\lambda I)^{\ast}f=0$ for all $y\in N\vee P_{\xi}$ and $f\in (N\vee P_{\xi})^{\perp}$. Therefore, $A-\lambda I\in\mathcal{J}(\mathcal{L},E)$ and then $\mathcal{R}\subset \mathbb{C}I+\mathcal{J}(\mathcal{L},E)$.

(2) Since $\mathcal{J}_{0}$ is a commutative Lie ideal, for any $A\in \mathcal{J}_{0}$, we have $[A,[A,T]]=0$ for all $T\in {\rm{Alg}}\mathcal{L}$. It follows from Proposition \ref{[A,[A,T]]=0} that there is $E\in \mathcal{L}$ such that $A\in \mathbb{C}I+\mathcal{J}(\mathcal{L},E)$, and by Remark \ref{CLI}, we can assume that $E=N\vee P_{\xi}$ with $N\in\mathcal{N}^{0}$. Now there are two mappings $u:\mathcal{J}_{0}\to \mathbb{C}$ and $v:\mathcal{J}_{0}\to \bigcup\{\mathcal{J}(\mathcal{L},N\vee P_{\xi}):N\in\mathcal{N}\}$ such that $A=u(A)I+v(A)$ for any $A\in\mathcal{J}_{0}$.

Hence, for all $A,B\in\mathcal{J}_{0}$, we have
$$
[v(A),[v(B),T]]=[A,[B,T]]=0, \ \forall T\in{\rm{Alg}}\mathcal{L},
$$
i.e.,
\begin{equation}\label{Th1.1}
v(A)v(B)T-v(A)Tv(B)-v(B)Tv(A)+Tv(B)v(A)=0.
\end{equation}
Let
$$
E=\vee \{N\vee P_{\xi}:N\in\mathcal{N},v(A)|_{N\vee P_{\xi}}=0, \ \forall A\in\mathcal{J}_{0}\}
$$
and
$$
N_{0}=\vee \{N\in\mathcal{N}:v(A)|_{N\vee P_{\xi}}=0, \ \forall A\in\mathcal{J}_{0}\}.
$$
Then $E=N_{0}\vee P_{\xi}$ and $v(A)|_{E}=0$ for all $A\in\mathcal{J}_{0}$. Since $\mathcal{J}_{0}$ is non-trivial Lie ideal, we have $E\neq I$ and $N_{0}\leq I_{-}^{\mathcal{N}}$.

Next, we prove that $v(A)^{\ast}|_{E^{\perp}}=0$ for all $A\in\mathcal{J}_{0}$. Let $A\in\mathcal{J}_{0}$. For $N>N_{0}$, we have $N\nleq N_{0}\vee P_{\xi}$, then there is $x\in N$ and $B\in \mathcal{J}_{0}$ such that $v(B)x\neq 0$. Taking $T=x\otimes f$ in Eq.(\ref{Th1.1}) for all $f\in N_{-}^{\perp}$ and noting $v(B)^{2}=0$, we get
$$
v(B)(x\otimes f)v(B)=0
$$
and
$$
v(A)v(B)(x\otimes f)-v(A)(x\otimes f)v(B)-v(B)(x\otimes f)v(A)+(x\otimes f)v(B)v(A)=0.
$$
The first equation gives $v(B)^{\ast}f=0$ for all $f\in N_{-}^{\perp}$ and the second equation becomes
$$
v(A)v(B)(x\otimes f)=v(B)(x\otimes f)v(A), \ \forall f\in N_{-}^{\perp}.
$$
Then $v(A)^{\ast}|_{N_{-}^{\perp}}$ is a scalar operator. By calculation, we have
\begin{align*}
  \vee\{N_{-}^{\perp}:N>N_{0}\} & = \vee\{(N_{-}^{\mathcal{N}}\vee P_{\xi})^{\perp}:N>N_{0}\}=(\wedge\{N_{-}^{\mathcal{N}}\vee P_{\xi}:N>N_{0}\})^{\perp}\\
   & =(\wedge\{N_{-}^{\mathcal{N}}:N>N_{0}\}\vee P_{\xi})^{\perp}=(N_{0}\vee P_{\xi})^{\perp}=E^{\perp}.
\end{align*}
Thus $v(A)^{\ast}|_{E^{\perp}}$ is a scalar operator, and since $(v(A)^{\ast}|_{E^{\perp}})^{2}=(v(A)^{2})^{\ast}|_{E^{\perp}}=0$, we have $v(A)^{\ast}|_{E^{\perp}}=0$.

By calculation, we have that $\mathcal{J}(\mathcal{L},N_{1}\vee P_{\xi})=\mathcal{J}(\mathcal{L},N_{2}\vee P_{\xi})$ if and only if $N_{1}=N_{2}$. This establishes the uniqueness. Indeed, if $\mathcal{J}(\mathcal{L},N_{1}\vee P_{\xi})=\mathcal{J}(\mathcal{L},N_{2}\vee P_{\xi})$ with $N_{1}\neq N_{2}$, then without loss of generality we may assume $N_{1}<N_{2}$. In this case, there exists $x\in N_{2}\setminus N_{1}$ and a nonzero element $f\in (N_{2}\vee P_{\xi})^{\perp}$ such that $x\otimes f\in\mathcal{J}(\mathcal{L},N_{2}\vee P_{\xi})$ but $x\otimes f\notin\mathcal{J}(\mathcal{L},N_{1}\vee P_{\xi})$.
\end{proof}

\section{The induced mapping}
When studying Lie isomorphisms on nest algebras over Banach spaces, Wang and Lu\cite{Lu4} introduced two mappings. We will adapt these to the setting of the algebra of one point extension of a nest and investigate their properties. Since some of these proof techniques are similar to those for nest algebras, we omit them.

Let $\psi$ be a Lie automorphism on ${\rm{Alg}}\mathcal{L}$. For $N\in\mathcal{N}^{0}$ with $E=N\vee P_{\xi}$, the maximal commutative Lie ideal $\mathbb{C}I+\mathcal{J}(\mathcal{L},E)$ in $\mathrm{Alg}\mathcal{L}$ has an image under $\psi$ that is again a maximal commutative Lie ideal in $\mathrm{Alg}\mathcal{L}$, and is uniquely represented as $\mathbb{C}I+\mathcal{J}(\mathcal{L},M\vee P_{\xi})$ for some unique $M\in\mathcal{N}^{0}$ by Theorem \ref{J(L,E)}.

Now we define
$$
\psi^{0}:\mathcal{N}^{0}\to\mathcal{N}^{0}
$$
by
$$
\psi(\mathbb{C}I+\mathcal{J}(\mathcal{L},N\vee P_{\xi}))=\mathbb{C}I+\mathcal{J}(\mathcal{L},\psi^{0}(N)\vee P_{\xi}).
$$
Then $\psi^{0}$ is a bijective mapping. 


For $A\in\mathcal{J}({\mathcal{L},N\vee P_{\xi}})$ with $N\in\mathcal{N}^{0}$, there is a unique operator $B\in\mathcal{J}(\mathcal{L},\psi^{0}(N)\vee P_{\xi})$ such that $\psi(A)-B\in\mathbb{C}I$. Thus we can define a mapping
$$
\phi:\bigcup\{\mathcal{J}(\mathcal{L},N\vee P_{\xi}):N\in\mathcal{N}^{0}\}\to\bigcup\{\mathcal{J}(\mathcal{L},N\vee P_{\xi}):N\in\mathcal{N}^{0}\}
$$
by
$$
\psi(A)-\phi(A)\in\mathbb{C}I.
$$
It is clear that $\phi$ is a linear and bijective, where the additivity$\phi(A+B)=\phi(A)+\phi(B)$ is required only for those $A+B$ belonging to $\bigcup\{\mathcal{J}(\mathcal{L},N\vee P_{\xi}):N\in\mathcal{N}^{0}\}$. Hence, we have $\phi(\mathcal{J}(\mathcal{L},N\vee P_{\xi}))=\mathcal{J}(\mathcal{L},\psi^{0}(N)\vee P_{\xi})$.
\begin{remark}\label{xi-eta}
For any $f\in(I_{-}^{\mathcal{N}}\vee P_{\xi})^{\perp}$, by Remark \ref{CLI}, we have $\xi\otimes f\in\mathcal{J}(\mathcal{L},N\vee P_{\xi})$ for any $N\in\mathcal{N}^{0}$. Since $\phi(\mathcal{J}(\mathcal{L},N\vee P_{\xi}))=\mathcal{J}(\mathcal{L},\psi^{0}(N)\vee P_{\xi})$, then there exist $T_{N}\in\mathcal{J}(\mathcal{L},\psi^{0}(N))$ and $g_{N}\in(I_{-}^{\mathcal{N}}\vee P_{\xi})^{\perp}$ such that
$$
\phi(\xi\otimes f)=T_{N}+\xi\otimes g_{N}, \ \forall N\in\mathcal{N}^{0}.
$$
Clearly, there are $T\in{\rm{Alg}}\mathcal{L}$ and $g\in(I_{-}^{\mathcal{N}}\vee P_{\xi})^{\perp}$ such that $T=T_{N}$ and $g=g_{N}$ for any $N\in\mathcal{N}^{0}$.

If $0_{+}^{\mathcal{N}}=0$, then $T\mathcal{H}=T_{N}\mathcal{H}\subset N$ for any $N\to 0$, which implies $T=0$. Thus $\phi(\xi\otimes f)=\xi\otimes g$.

If $0_{+}^{\mathcal{N}}\neq0$, then since $T\mathcal{H}=T_{0_{+}^{\mathcal{N}}}\mathcal{H}\subset 0_{+}^{\mathcal{N}}$ and $TI_{-}^{\mathcal{N}}=0$, by the decomposition $\mathcal{H}=0_{+}^{\mathcal{N}}\oplus (I_{-}^{\mathcal{N}}\ominus 0_{+}^{\mathcal{N}})\oplus (I_{-}^{\mathcal{N}})^{\perp}$, we have
$$
T=\left(\begin{array}{ccc}
0&0&T_{13}\\
0&0&0\\
0&0&0
\end{array}\right)\in{\rm{Alg}}\mathcal{L}.
$$
\end{remark}

\begin{lemma}\label{order}
Let $N_{1},N_{2},N_{3}\in\mathcal{N}^{0}$ with $N_{1}<N_{2}<N_{3}$. Suppose that one of $\psi^{0}(N_{1})<\psi^{0}(N_{2})$, $\psi^{0}(N_{1})<\psi^{0}(N_{3})$ and $\psi^{0}(N_{2})<\psi^{0}(N_{3})$ holds. Then $\psi^{0}(N_{1})<\psi^{0}(N_{2})<\psi^{0}(N_{3})$.
\end{lemma}
\begin{proof}
For $i\in \{1,2,3\}$, take $x_{i}\in N_{i}$ and $f_{i}\in(N_{i}\vee P_{\xi})^{\perp}$ such that $\langle x_{2},f_{1}\rangle=\langle x_{3},f_{2}\rangle=1$. Then $x_{i}\otimes f_{i}\in\mathcal{J}(\mathcal{L},N_{i})\subset\mathcal{J}(\mathcal{L},N_{i}\vee P_{\xi})$, therefore, $\phi(x_{i}\otimes f_{i})\in\mathcal{J}(\mathcal{L},\psi^{0}(N_{i})\vee P_{\xi})$. So for each case of $\psi^{0}(N_{1})<\psi^{0}(N_{2})$, $\psi^{0}(N_{1})<\psi^{0}(N_{3})$, and $\psi^{0}(N_{2})<\psi^{0}(N_{3})$, since $\psi^{0}(N_{i})\vee P_{\xi}$ is the invariant subspace of any operator in ${\rm{Alg}}\mathcal{L}$, then $\phi(x_{3}\otimes f_{3})\phi(x_{2}\otimes f_{2})\phi(x_{1}\otimes f_{1})=0$. By calculation, we have
\begin{align}\label{rank-1}
  \psi(x_{1}\otimes f_{3}) &= \psi([x_{1}\otimes f_{1},[x_{2}\otimes f_{2},x_{3}\otimes f_{3}]]) \notag\\
  &= [\phi(x_{1}\otimes f_{1}),[\phi(x_{2}\otimes f_{2}),\phi(x_{3}\otimes f_{3})]]\notag\\
  &= \phi(x_{1}\otimes f_{1})\phi(x_{2}\otimes f_{2})\phi(x_{3}\otimes f_{3})\notag\\
  &-\phi(x_{1}\otimes f_{1})\phi(x_{3}\otimes f_{3})\phi(x_{2}\otimes f_{2})-\phi(x_{2}\otimes f_{2})\phi(x_{3}\otimes f_{3})\phi(x_{1}\otimes f_{1})
\end{align}
and
\begin{align}\label{rank-2}
  \psi(x_{1}\otimes f_{3}) &= \psi([x_{3}\otimes f_{3},[x_{2}\otimes f_{2},x_{1}\otimes f_{1}]]) \notag\\
  &= [\phi(x_{3}\otimes f_{3}),[\phi(x_{2}\otimes f_{2}),\phi(x_{1}\otimes f_{1})]]\notag\\
  &= \phi(x_{1}\otimes f_{1})\phi(x_{2}\otimes f_{2})\phi(x_{3}\otimes f_{3})\notag\\
  &-\phi(x_{3}\otimes f_{3})\phi(x_{1}\otimes f_{1})\phi(x_{2}\otimes f_{2})-\phi(x_{2}\otimes f_{2})\phi(x_{1}\otimes f_{1})\phi(x_{3}\otimes f_{3}).
\end{align}
Next, we will prove that
$$
\phi(x_{1}\otimes f_{1})\phi(x_{3}\otimes f_{3})\phi(x_{2}\otimes f_{2})=0
$$
and
$$
\phi(x_{2}\otimes f_{2})\phi(x_{3}\otimes f_{3})\phi(x_{1}\otimes f_{1})=0.
$$
If $\phi(x_{1}\otimes f_{1})\phi(x_{3}\otimes f_{3})\phi(x_{2}\otimes f_{2})\neq 0$, then $\phi(x_{3}\otimes f_{3})\phi(x_{2}\otimes f_{2})\neq 0$, and
$$
\psi^{0}(N_{2})\vee P_{\xi}\nleq \psi^{0}(N_{3})\vee P_{\xi}.
$$
This implies $\psi^{0}(N_{2})\vee P_{\xi}>\psi^{0}(N_{3})\vee P_{\xi}$. Similarly, $\psi^{0}(N_{3})\vee P_{\xi}>\psi^{0}(N_{1})\vee P_{\xi}$. Thus
$$
\psi^{0}(N_{2})>\psi^{0}(N_{3})>\psi^{0}(N_{1}).
$$
By Eq. (\ref{rank-2}), this leads to $\psi(x_{1}\otimes f_{3})=0$, a contradiction.

If $\phi(x_{2}\otimes f_{2})\phi(x_{3}\otimes f_{3})\phi(x_{1}\otimes f_{1})\neq0$, then $\psi^{0}(N_{1})>\psi^{0}(N_{3})>\psi^{0}(N_{2})$, and again by Eq. (\ref{rank-2}), we have $\psi(x_{1}\otimes f_{3})=0$, a contradiction. Therefore, the claim is holds and by Eq. (\ref{rank-1}), we have
$$
\psi(x_{1}\otimes f_{3})=\phi(x_{1}\otimes f_{1})\phi(x_{2}\otimes f_{2})\phi(x_{3}\otimes f_{3})\neq0.
$$
This implies
$$
\psi^{0}(N_{1})<\psi^{0}(N_{2})<\psi^{0}(N_{3}).
$$
\end{proof}

\begin{proposition}\label{order}
$\psi^{0}$ is either order-preserving or order-reversing.
\end{proposition}
\begin{proof}
Since the proof is similar to that for nest algebras, we omit it.\end{proof}

Recall that ${\rm{idem}}(\mathcal{L})$ denote the sets of all idempotent operators in ${\rm{Alg}}\mathcal{L}$. For any $A\in{\rm{idem}}(\mathcal{L})\setminus\{0,I\}$, by Proposition \ref{idem}, there exist $\lambda\in\mathbb{C}$ and $B\in{\rm{idem}}(\mathcal{L})$ such that $\psi(A)=\lambda I+B$. Next we prove that $B\neq 0,I$. Otherwise, $\psi(A)\in\mathbb{C}I$, since $\psi$ is a bijective mapping and $\psi(I)\in\mathbb{C}I$, there is some scalar $\mu\in\mathbb{C}I$ such that $A=\mu I$. As $A$ is an idempotent, $\mu=0$ or $\mu=1$, implying $A=0$ or $A=I$, a contradiction. Thus $B\in{\rm{idem}}(\mathcal{L})\setminus\{0,I\}$. Now, we can define the mapping
$$
\varphi:{\rm{idem}}(\mathcal{L})\setminus\{0,I\}\to{\rm{idem}}(\mathcal{L})\setminus\{0,I\}
$$
by
$$
\psi(A)-\varphi(A)\in\mathbf{C}I, \forall \ A\in{\rm{idem}}(\mathcal{L})\setminus\{0,I\}.
$$
A direct verification shows that $\varphi$ is a bijective mapping. We extend $\varphi$ to all idempotents by setting $\varphi(0)=0$ and $\varphi(I)=I$, and we denote the extended map again by $\varphi$.

At the end of this section, we introduce two commonly used set notations in a remark.

\begin{remark}
Let $E\in\mathcal{L}^{0}$. Define
$$
\Omega_{1}(\mathcal{L},E)=\{P\in{\rm{idem}}(\mathcal{L}):PE=0, \ P^{\ast}E^{\perp}\neq0\}
$$
and
$$
\Omega_{2}(\mathcal{L},E)=\{P\in{\rm{idem}}(\mathcal{L}):PE\neq0, \ P^{\ast}E^{\perp}=0\}.
$$
By definition, the set $\Omega_{2}(\mathcal{L},E)$ is monotonically increasing relative to $E\in\mathcal{L}^{0}$ while $\Omega_{1}(\mathcal{L},E)$ is monotonically decreasing. Moreover, $\Omega_{1}(\mathcal{L},N\vee P_{\xi})\subsetneq\Omega_{1}(\mathcal{L},N)$. Indeed, for any $N\in\mathcal{N}^{0}$, consider the decomposition $\mathcal{H}=N\oplus N^{\perp}$. Write $\xi=(\xi_{1},\xi_{2})^{T}$, then
$$
\left(\begin{array}{cc}
0&\frac{\xi_{1}\otimes \xi_{2}}{\|\xi_{2}\|^{2}}\\
0&1
\end{array}\right)\in\Omega_{1}(\mathcal{L},N)\setminus\Omega_{1}(\mathcal{L},N\vee P_{\xi}).
$$

For any $f\in (I_{-}^{\mathcal{N}})^{\perp}$ with $\langle \xi,f\rangle=1$, we have $\xi\otimes f\in{\rm{Alg}}\mathcal{L}$ is an idempotent. For any $N\in\mathcal{N}^{0}$, by calculation, we have $\xi\otimes f\in \Omega_{2}(\mathcal{L},N\vee P_{\xi})\setminus\Omega_{2}(\mathcal{L},N)$. Consequently, we obtain $\Omega_{2}(\mathcal{L},N)\subsetneq\Omega_{2}(\mathcal{L},N\vee P_{\xi})$.

\end{remark}
\section{Characterizations for the Main Results}


Now we consider the structure of $\Omega_{2}(\mathcal{L},E)$ under the mapping $\varphi$.

Let $P\in \Omega_{2}(\mathcal{L},E)$. For each $C\in\mathcal{J}(\mathcal{L},E)$, we have $C\mathcal{H}\subset E$ and $CE=0$, hence $CP=0$ and
$$
[P,[P,C]]=[P,PC]=PC=[P,C].
$$
Thus
$$
[\psi(P),[\psi(P),\psi(C)]]=[\psi(P),\psi(C)], \ \forall C\in\mathcal{J}(\mathcal{L},E).
$$
Since $\psi(\mathbb{C}I+\mathcal{J}(\mathcal{L},E))=\mathbb{C}I+\mathcal{J}(\mathcal{L},\psi^{0}(N)\vee P_{\xi})$, we have
$$
[\varphi(P),[\varphi(P),D]]=[\varphi(P),D], \ \forall D\in\mathcal{J}(\mathcal{L},\psi^{0}(N)\vee P_{\xi}).
$$
By a direct calculation and $\varphi(P)$ an idempotent, we have
\begin{equation}\label{D}
\varphi(P)D\varphi(P)=D\varphi(P), \ \forall D\in\mathcal{J}(\mathcal{L},\psi^{0}(N)\vee P_{\xi}).
\end{equation}

\begin{lemma}\label{3-cases}
Let $N\in\mathcal{N}^{0}$, $E=N\vee P_{\xi}$, and $P\in\Omega_{2}(\mathcal{L},E)$. Then $\varphi(P)$ falls into exactly three cases.
\end{lemma}

\begin{proof}
(1) If $\varphi^{\ast}(P)(\psi^{0}(N)\vee P_{\xi})^{\perp}=0$, by the composition $\mathcal{H}=(\psi^{0}(N)\vee P_{\xi})\oplus (\psi^{0}(N)\vee P_{\xi})^{\perp}$, we have
\begin{equation}
\varphi(P)=\left(\begin{array}{cc}
\tilde{P}_{11}&\tilde{P}_{12}\\
0&0
\end{array}\right).
\tag{Form (1)}
\end{equation}

(2) If $\varphi^{\ast}(P)(\psi^{0}(N)\vee P_{\xi})^{\perp}\neq0$, then there exists $f\in(\psi^{0}(N)\vee P_{\xi})^{\perp}$ such that $\varphi^{\ast}(P)f\neq 0$. For any $x\in \psi^{0}(N)$, we have $x\otimes f\in\mathcal{J}(\mathcal{L},\psi^{0}(N)\vee P_{\xi})$, and by Eq.(\ref{D}), we obtain
$$
\varphi(P)(x\otimes f)\varphi(P)=(x\otimes f)\varphi(P).
$$
Thus $\varphi(P)x=x$ for any $x\in \psi^{0}(N)$.

(2.1) If $\varphi(P)\xi=\xi$, by the composition $\mathcal{H}=(\psi^{0}(N)\vee P_{\xi})\oplus (\psi^{0}(N)\vee P_{\xi})^{\perp}$, it follows that
\begin{equation}
\varphi(P)=\left(\begin{array}{cc}
1&\tilde{P}_{12}\\
0&\tilde{P}_{22}
\end{array}\right),
\tag{Form (2)}
\end{equation}
where $\tilde{P}_{22}\neq 0,1$. 

(2.2) If $\varphi(P)\xi=0$. For any $g\in (I_{-}^{\mathcal{N}}\vee P_{\xi})^{\perp}$, we have $\xi\otimes g\in \mathcal{J}(\mathcal{L},\psi^{0}(N)\vee P_{\xi})$. By Eq.(\ref{D}), we have
$$
0=\varphi(P)(\xi\otimes g)\varphi(P)=(\xi\otimes g)\varphi(P),
$$
it follows that $\varphi^{\ast}(P)(I_{-}^{\mathcal{N}}\vee P_{\xi})^{\perp}=0$ and $\varphi(P)$ is an idempotent, by the composition $\mathcal{H}=\psi^{0}(N)\oplus (I_{-}^{\mathcal{N}}\ominus \psi^{0}(N))\oplus (I_{-}^{\mathcal{N}})^{\perp}$, we obtain
\begin{equation}
\varphi(P)=\left(\begin{array}{ccc}
1&\tilde{P}_{12}&\tilde{P}_{13}\\
0&\tilde{P}_{22}&\tilde{P}_{23}\\
0&0&0
\end{array}\right),
\tag{Form (3)}
\end{equation}
where $\tilde{P}_{22}\neq 0$.

Moreover, it can be observed that when $\psi^{0}(N)= I_{-}^{\mathcal{N}}$, $\varphi(P)$ certainly does not satisfy Form (3).
\end{proof}

By Remark \ref{xi-eta}, there are different structures corresponding to different $0_{+}^{\mathcal{N}}$. We shall therefore consider the following two cases.

\subsection{The case that $0_{+}^{\mathcal{N}}=0$}

For any $f\in(I_{-}^{\mathcal{N}}\vee P_{\xi})^{\perp}$, it follows from Remark \ref{xi-eta} that there is some $\tilde{f}\in(I_{-}^{\mathcal{N}}\vee P_{\xi})^{\perp}$ such that $\phi(\xi\otimes f)=\xi\otimes \tilde{f}$. Consider the mapping $\psi^{-1}$, we have $\phi(\{\xi\otimes f:f\in(I_{-}^{\perp}\vee P_{\xi})^{\perp}\})=\{\xi\otimes \tilde{f}:\tilde{f}\in(I_{-}^{\mathcal{N}}\vee P_{\xi})^{\perp}\}$.

For any $f_{0}\in (I_{-}^{\mathcal{N}})^{\perp}$ with $\langle \xi, f_{0}\rangle=1$, the operator $\xi\otimes f_{0}$ is an idempotent. We claim that $\varphi(\xi\otimes f_{0})\xi=\xi$. Indeed, suppose to the contrary that $\varphi(\xi\otimes f_{0})\xi\neq \xi$. Since $\varphi(\xi\otimes f_{0})$ is an idempotent, it follows that $\varphi(\xi\otimes f_{0})\xi=0$. By direct computation, we have
$$
[\xi\otimes f_{0},\xi\otimes f]=\xi\otimes f, \ \forall f\in(I_{-}^{\mathcal{N}}\vee P_{\xi})^{\perp}.
$$
Therefore,
$$
[\varphi(\xi\otimes f_{0}),\phi(\xi\otimes f)]=\phi(\xi\otimes f), \ \forall f\in(I_{-}^{\mathcal{N}}\vee P_{\xi})^{\perp},
$$
this implies
\begin{align}\label{ast}
[\varphi(\xi\otimes f_{0}),\xi\otimes \tilde{f}]=\xi\otimes \tilde{f}, \ \forall \tilde{f}\in(I_{-}^{\mathcal{N}}\vee P_{\xi})^{\perp}.
\end{align}
Since $\varphi(\xi\otimes f_{0})\xi=0$, it follows that $\varphi(\xi\otimes f_{0})^{\ast}\tilde{f}=\tilde{f}$ for all $\tilde{f}\in(I_{-}^{\mathcal{N}}\vee P_{\xi})^{\perp}$, which contradicts Lemma \ref{3-cases}. This completes the proof of the claim. Thus $\varphi(\xi\otimes f_{0})\xi=\xi$ and then $\varphi(\xi\otimes f_{0})$ not have the Form (3), and by Eq. (\ref{ast}), we have $\varphi^{\ast}(\xi\otimes f_{0})(I_{-}^{\mathcal{N}}\vee P_{\xi})^{\perp}=0$.

\begin{lemma}\label{xi-f0}
For any $f_{0}\in (I_{-}^{\mathcal{N}})^{\perp}$ with $\langle \xi, f_{0}\rangle=1$, there exists some $\tilde{f}_{0} \in (I_{-}^{\mathcal{N}})^{\perp}$ with $\langle \xi, \tilde{f}_{0} \rangle = 1$ such that $\varphi(\xi \otimes f_{0}) = \xi \otimes \tilde{f}_{0}$.
\end{lemma}

\begin{proof}
Firstly, we prove that there does not exist any $N \in \mathcal{N}^{0}$ such that $\varphi(\xi \otimes f_{0})$ satisfies Form (2) in Lemma \ref{3-cases}.

Suppose there exists some $N \in \mathcal{N}^{0}$ such that $\varphi(\xi \otimes f_{0})$ has Form (2). Let $N_{0}$ be the supremum of all $N \in \mathcal{N}^{0}$ for which $\varphi(\xi \otimes f_{0})$ has Form (2) under the decomposition $\mathcal{H}=(\psi^{0}(N)\vee P_{\xi})\oplus (\psi^{0}(N)\vee P_{\xi})^{\perp}$.

Now, for $\varphi(\xi \otimes f_{0})$, there are two cases under the decomposition
$$
\mathcal{H}=(\psi^{0}(N_{0})\vee P_{\xi})\oplus (\psi^{0}(N_{0})\vee P_{\xi})^{\perp},
$$
either it has Form (1) or it has Form (2). In the following, we show that neither case can occur.

If $\varphi(\xi\otimes f_{0})$ has the Form (1) by the decomposition $\mathcal{H}=(\psi^{0}(N_{0})\vee P_{\xi})\oplus (\psi^{0}(N_{0})\vee P_{\xi})^{\perp}$, this implies
$$
\varphi(\xi\otimes f_{0})=\left(\begin{array}{cc}
1&*\\
0&0
\end{array}\right)=
\left(\begin{array}{cc}
1&\tilde{P}_{12}\\
0&0
\end{array}\right)+\xi\otimes f, \ {\rm{for \ some \ }}f\in(I_{-}^{\mathcal{N}})^{\perp},
$$
where the second matrix by the decomposition $\mathcal{H}=\psi^{0}(N_{0})\oplus \psi^{0}(N_{0})^{\perp}$.

For any $C\in\mathcal{J}(\mathcal{L},N_{0})$, we have $[\xi\otimes f_{0},C]=0$, this implies $[\varphi(\xi\otimes f_{0}),D]=0$ for some $D\in\mathcal{J}(\mathcal{L},\psi^{0}(N_{0})\vee P_{\xi})$. By Remark \ref{CLI}, we have
$$
D=\left(\begin{array}{cc}
0&D_{12}\\
0&0
\end{array}\right)+\xi\otimes f^{'}, \ {\rm{for \ some \ }}f^{'}\in(I_{-}^{\mathcal{N}}\vee P_{\xi})^{\perp},
$$
where $D_{12}\neq 0$. By direct calculation, we have $0=[\varphi(\xi\otimes f_{0}),D]=D$, which leads to a contradiction.

If $\varphi(\xi\otimes f_{0})$ has the Form (2) with $\mathcal{H}=(\psi^{0}(N_{0})\vee P_{\xi})\oplus (\psi^{0}(N_{0})\vee P_{\xi})^{\perp}$. By Proposition \ref{order}, the mapping $\psi^{0}$ is order-preserving, which implies $\psi^{0}(I_{-}^{\mathcal{N}}) = I_{-}^{\mathcal{N}}$. Since $\varphi^{*}(\xi \otimes f_{0})(I_{-}^{\mathcal{N}} \vee P_{\xi})^{\perp} = 0$, it follows that $\varphi(\xi \otimes f_{0})$ does not satisfy Form (2) under the decomposition $\mathcal{H}=(I_{-}^{\mathcal{N}}\vee P_{\xi})\oplus (I_{-}^{\mathcal{N}}\vee P_{\xi})^{\perp}$. This implies $N_{0}<I_{-}^{\mathcal{N}}$ and $\psi^{0}(N_{0})<I_{-}^{\mathcal{N}}$. Now, we shall show that $(N_{0})^{\mathcal{N}}_{+}\neq N_{0}$. If $(N_{0})^{\mathcal{N}}_{+}= N_{0}$, then there is a sequence $N\in\mathcal{N}^{0}$ with descending to $N_{0}$ such that $\varphi(\xi\otimes f_{0})$ have the Form (1) by the decomposition $\mathcal{H}=(\psi^{0}(N)\vee P_{\xi})\oplus (\psi^{0}(N)\vee P_{\xi})^{\perp}$. Consequently, $\varphi(\xi\otimes f_{0})$ would also have the Form (1) by the decomposition $\mathcal{H}=(\psi^{0}(N_{0})\vee P_{\xi})\oplus (\psi^{0}(N_{0})\vee P_{\xi})^{\perp}$, which leads to a contradiction. Thus we can assume that
$$
\varphi(\xi\otimes f_{0})=\left(\begin{array}{ccc}
1&*&*\\
0&(\neq0)&*\\
0&0&0
\end{array}\right)
$$
with the decomposition $\mathcal{H}=(\psi^{0}(N_{0})\vee P_{\xi})\oplus \{(\psi^{0}((N_{0})^{\mathcal{N}}_{+})\vee P_{\xi})\ominus (\psi^{0}(N_{0})\vee P_{\xi})\}\oplus (\psi^{0}((N_{0})^{\mathcal{N}}_{+})\vee P_{\xi})^{\perp}$. In fact, we may write in the decomposition $\mathcal{H}=\psi^{0}({N_{0}})\oplus (\psi^{0}((N_{0})^{\mathcal{N}}_{+})\ominus \psi^{0}({N_{0}}))\oplus (\psi^{0}((N_{0})^{\mathcal{N}}_{+}))^{\perp}$, under which
$$
\varphi(\xi\otimes f_{0})=\left(\begin{array}{ccc}
1&\tilde{P}_{12}&\tilde{P}_{13}\\
0&\tilde{P}_{22}&\tilde{P}_{23}\\
0&0&0
\end{array}\right)+\xi\otimes f,
$$
where $f\in(I_{-}^{\mathcal{N}})^{\perp}$ satisfies $\langle \xi,f\rangle=1$. Similarly, since $[\xi\otimes f_{0},C]=0$ for all $C\in\mathcal{J}(\mathcal{L},N_{0})$, it follows that $[\varphi(\xi\otimes f_{0}),D]=0$ for all $D\in\mathcal{J}(\mathcal{L},\psi^{0}(N_{0}))$. This leads to a contradiction. Indeed, taking
$$
D=\left(\begin{array}{ccc}
0&0&D_{13}\\
0&0&0\\
0&0&0
\end{array}\right),
$$
with $D_{13}\neq 0$, we have $[\varphi(\xi\otimes f_{0}),D]\neq 0$.

In summary, $\varphi(\xi \otimes f_{0})$ satisfies Form (1) for every $N \in \mathcal{N}^{0}$. Since $\varphi(\xi \otimes f_{0})\mathcal{H} \subset \psi^{0}(N) \vee P_{\xi}$, it follows that $\varphi(\xi \otimes f_{0})\mathcal{H} \subset P_{\xi}$. Therefore, there exists some $\tilde{f}_{0} \in (I_{-}^{\mathcal{N}})^{\perp}$ with $\langle \xi, \tilde{f}_{0} \rangle = 1$ such that $\varphi(\xi \otimes f_{0}) = \xi \otimes \tilde{f}_{0}$.
\end{proof}

\begin{corollary}\label{J(L,N)}
Let $N\in\mathcal{N}^{0}$ and $C\in\mathcal{J}(\mathcal{L},N)$. Then $\phi(C)\in\mathcal{J}(\mathcal{L},\psi^{0}(N))$. Moreover, $\phi(\mathcal{J}(\mathcal{L},N))=\mathcal{J}(\mathcal{L},\psi^{0}(N))$.
\end{corollary}

\begin{proof}
For any $C\in\mathcal{J}(\mathcal{L},N)$, since $\phi(\mathcal{J}(\mathcal{L},N\vee P_{\xi}))=\mathcal{J}(\mathcal{L},\psi^{0}(N)\vee P_{\xi})$, so there are $D\in\mathcal{J}(\mathcal{L},\psi^{0}(N))$ and $f\in(I_{-}^{\mathcal{N}}\vee P_{\xi})^{\perp}$ such that $\phi(C)=D+\xi\otimes f$. By calculation, we have $[\xi\otimes f_{0},C]=0$ for any $f_{0}\in (I_{-}^{\mathcal{N}})^{\perp}$ with $\langle \xi, f_{0}\rangle=1$, it follows that $[\varphi(\xi\otimes f_{0}),\phi(C)]=0$. It follow from Lemma \ref{xi-f0} that there exists some $\tilde{f}_{0}\in (I_{-}^{\mathcal{N}})^{\perp}$ with $\langle \xi, \tilde{f}_{0}\rangle = 1$ such that $\varphi(\xi \otimes f_{0}) = \xi \otimes \tilde{f}_{0}$, thus
$$
[\xi\otimes \tilde{f}_{0},D+\xi\otimes f]=0.
$$
This implies $f=0$, hence $\phi(C)=D\in\mathcal{J}(\mathcal{L},\psi^{0}(N))$. By considering $\phi^{-1}$, we obtain $\phi(\mathcal{J}(\mathcal{L},N))=\mathcal{J}(\mathcal{L},\psi^{0}(N))$.
\end{proof}

\begin{proposition}\label{Omega1}
For any $N\in\mathcal{N}^{0}$ and $Q \in \Omega_{2}(\mathcal{L}, N)$, we have $\varphi(Q)\in \Omega_{2}(\mathcal{L}, \psi^{0}(N))$.
\end{proposition}
\begin{proof}
Let $0<N<I_{-}^{\mathcal{N}}$. With the decomposition $\mathcal{H}=N\oplus (I_{-}^{\mathcal{N}}\ominus N)\oplus (I_{-}^{\mathcal{N}})^{\perp}$, we first consider the special idempotent
$$
Q=\left(\begin{array}{ccc}
1&Q_{12}&Q_{13}\\
0&0&0\\
0&0&0
\end{array}\right)\in\Omega_{2}(\mathcal{L}, N).
$$
By Lemma \ref{3-cases}, $\varphi(Q)$ can take one of three forms. Suppose $\varphi(Q)$ satisfies Form (2). For any $f\in(I_{-}^{\mathcal{N}})^{\perp}$, we have $[Q,\xi\otimes f]=0$. By Remark \ref{xi-eta} and Lemma \ref{xi-f0}, it follows that $[\varphi(Q),\xi\otimes f]=0$ for all $f\in(I_{-}^{\mathcal{N}})^{\perp}$. Since $\varphi(Q)\xi=\xi$, we obtain $\varphi^{\ast}(Q)(I_{-}^{\mathcal{N}})^{\perp}=1$. Hence, under the decomposition $\mathcal{H} = \psi^{0}(N) \oplus (I_{-}^{\mathcal{N}} \ominus \psi^{0}(N)) \oplus (I_{-}^{\mathcal{N}})^{\perp}$, we may assume
$$
\varphi(Q)=\left(\begin{array}{ccc}
1&\tilde{Q}_{12}&\tilde{Q}_{13}\\
0&\tilde{Q}_{22}&\tilde{Q}_{23}\\
0&0&1
\end{array}\right).
$$
Now consider $[Q,C]$ for any $C\in\mathcal{J}(\mathcal{L},I_{-}^{\mathcal{N}})$. This implies $\tilde{Q}_{22}=1$, and therefore $\varphi(Q)=I$. Consequently $Q=I$, which is a contradiction.

If $\varphi(Q)$ satisfies Form (3), then, under the decomposition $\mathcal{H} = \psi^{0}(N) \oplus (I_{-}^{\mathcal{N}} \ominus \psi^{0}(N)) \oplus (I_{-}^{\mathcal{N}})^{\perp}$, we have
$$
\varphi(Q)=\left(\begin{array}{ccc}
1&\tilde{Q}_{12}&\tilde{Q}_{13}\\
0&\tilde{Q}_{22}(\neq0)&Q_{23}\\
0&0&0
\end{array}\right).
$$
Since $[Q,C]=C$ for all $C\in\mathcal{J}(\mathcal{L},N)$, it follows Corollary \ref{J(L,N)} that $[\varphi(Q),\tilde{C}]=\tilde{C}$ for all $\tilde{C}\in\mathcal{J}(\mathcal{L},\psi^{0}(N))$. This implies 
$\tilde{Q}_{22}=0$. This is a contradiction.

Therefore, $\varphi(Q)$ must be of Form (1), under the decomposition $\mathcal{H} = \psi^{0}(N) \oplus (I_{-}^{\mathcal{N}} \ominus \psi^{0}(N)) \oplus (I_{-}^{\mathcal{N}})^{\perp}$, we have
$$
\varphi(Q)
=\left(\begin{array}{ccc}
\tilde{Q}_{11}&\tilde{Q}_{12}&\tilde{Q}_{13}\\
0&0&0\\
0&0&0
\end{array}\right)+\xi\otimes f, \ {\rm{for \ some \ }}f\in(I_{-}^{\mathcal{N}})^{\perp}.
$$
Since $[Q,C]=C$ for all $C\in\mathcal{J}(\mathcal{L},N)$, it follows 
$$
\varphi(Q)
=\left(\begin{array}{ccc}
1&\tilde{Q}_{12}&\tilde{Q}_{13}\\
0&0&0\\
0&0&0
\end{array}\right)+\xi\otimes f.
$$

If there exists an element $Q$ such that the term $\xi \otimes f$ in the expression $\varphi(Q) = \tilde{Q} + \xi \otimes f$ is nonzero, then $\langle \xi, f \rangle = 1$. However, since $[Q,\xi\otimes g]=0$ for all $g\in (I_{-}^{\mathcal{N}}\vee P_{\xi})^{\perp}$, but $[\varphi(Q),\xi\otimes g]=\xi\otimes g\neq 0$, we obtain a contradiction.
Therefore, the mapping $\varphi$ acts as
$$
\left(\begin{array}{ccc}
1&Q_{12}&Q_{13}\\
0&0&0\\
0&0&0
\end{array}\right)
\mapsto\left(\begin{array}{ccc}
1&\tilde{Q}_{12}&\tilde{Q}_{13}\\
0&0&0\\
0&0&0
\end{array}\right).
$$

Now we consider a general idempotent in $\Omega_{2}(\mathcal{L}, N)$. Let
$$
Q=\left(\begin{array}{ccc}
Q_{11}(\neq 1)&Q_{12}&Q_{13}\\
0&0&0\\
0&0&0
\end{array}\right).
$$
By Lemma \ref{3-cases}, $\varphi(Q)$ can take one of three forms. Similarly, we can prove that the Form (2) does not exists. 
Now if $\varphi(Q)$ satisfies Form (3), then
$$
\varphi(Q)=\left(\begin{array}{ccc}
1&\tilde{Q}_{12}&\tilde{Q}_{13}\\
0&\tilde{Q}_{22}(\neq0)&\tilde{Q}_{23}\\
0&0&0
\end{array}\right).
$$
Thus there exists an idempotent
$$
Q'=\left(\begin{array}{ccc}
1&*&*\\
0&0&0\\
0&0&0
\end{array}\right) {\rm{ \ such \ that \ }}
\varphi(Q)-\varphi(Q')=\left(\begin{array}{ccc}
0&0&0\\
0&\tilde{Q}_{22}&\tilde{Q}_{23}\\
0&0&0
\end{array}\right).
$$
Since $Q_{11}\neq 1$, there exists some $x\in\psi^{0}(N)$ such that $(Q-Q')x\neq 0$, and hence $(Q-Q')(x\otimes f)\neq0$ for any $f\in (I_{-}^{\mathcal{N}}\vee P_{\xi})^{\perp}$. Since $x\otimes f\in\mathcal{J}(\mathcal{L},N)\cap \mathcal{J}(\mathcal{L},I_{-}^{\mathcal{N}})$, we have $\phi(x\otimes f)\in\mathcal{J}(\mathcal{L},\psi^{0}(N))\cap \mathcal{J}(\mathcal{L},I_{-}^{\mathcal{N}})$. Therefore, we have $[Q-Q', x\otimes f]\neq0$, but $[\varphi(Q)-\varphi(Q'), \psi(x\otimes f)]=0$, which leads to a contradiction.

Therefore, $\varphi(Q)$ satisfies Form (1). Thus, under the decomposition $\mathcal{H}= \psi^{0}(N) \oplus (I_{-}^{\mathcal{N}} \ominus \psi^{0}(N)) \oplus (I_{-}^{\mathcal{N}})^{\perp}$, we have
$$
\varphi(Q)
=\left(\begin{array}{ccc}
\tilde{Q}_{11}&\tilde{Q}_{12}&\tilde{Q}_{13}\\
0&0&0\\
0&0&0
\end{array}\right)+\xi\otimes f, \ {\rm{for \ some \ }} f\in (I_{-}^{\mathcal{N}})^{\perp}.
$$
Similarly, we can prove that $\xi\otimes f=0$, hence the mapping $\varphi$ acts as
$$
\left(\begin{array}{ccc}
Q_{11}&Q_{12}&Q_{13}\\
0&0&0\\
0&0&0
\end{array}\right)
\mapsto\left(\begin{array}{ccc}
\tilde{Q}_{11}&\tilde{Q}_{12}&\tilde{Q}_{13}\\
0&0&0\\
0&0&0
\end{array}\right)\in \Omega_{2}(\mathcal{L}, \psi^{0}(N)).
$$

We begin by assuming that $0<N<I_{-}^{\mathcal{N}}$. Note that the above proof remains valid when $N=I_{-}^{\mathcal{N}}$.
\end{proof}
\begin{corollary}\label{=J(L,E)}
Let $N\in\mathcal{N}^{0}$. Then $\psi(C)=\phi(C)$ for all $C\in\mathcal{J}(\mathcal{L},N)$, and $\psi(\xi\otimes f)=\phi(\xi\otimes f)$ for all $f\in (I^{\mathcal{N}}_{-}\vee P_{\xi})^{\perp}$, i.e., $\psi(D)=\phi(D)$ for all $D\in \mathcal{J}(\mathcal{L},N\vee P_{\xi})$.
\end{corollary}
\begin{proof}
For any $C\in\mathcal{J}(\mathcal{L},N)$, we have $C=[Q,C]$ for all $
Q=\left(\begin{array}{ccc}
1&*&*\\
0&0&0\\
0&0&0
\end{array}\right)
$, it follows
$$
\psi(C)=[\psi(Q),\psi(C)]=[\varphi(Q),\phi(C)]=\phi(C).
$$

For any $\xi\otimes f\in\mathcal{J}(\mathcal{L},N\vee P_{\xi})$, we have $\xi\otimes f=[\xi\otimes f_{0},\xi\otimes f]$ for all $f_{0}\in (I^{\mathcal{N}}_{-})^{\perp}$ with $\langle \xi,f_{0}\rangle=1$, it follow from Remark \ref{xi-eta} and Lemma \ref{xi-f0} that
$$
\psi(\xi\otimes f)=[\psi(\xi\otimes f_{0}),\psi(\xi\otimes f)]=[\varphi(\xi\otimes f_{0}),\phi(\xi\otimes f)]=\phi(\xi\otimes f).
$$
\end{proof}

\subsection{ The case that $0_{+}^{\mathcal{N}}\neq 0$}
\begin{remark}\label{traverse}
Since $\mathcal{N}$ is a nest with at least four projections, we obtain $0_{+}^{\mathcal{N}}\neq I_{-}^{\mathcal{N}}$. By Theorem \ref{J(L,E)} and Remark \ref{xi-eta}, under the decomposition $\mathcal{H} = 0_{+}^{\mathcal{N}} \oplus (I_{-}^{\mathcal{N}} \ominus 0_{+}^{\mathcal{N}}) \oplus (I_{-}^{\mathcal{N}})^{\perp}$, we have
$$
\phi:\left(\begin{array}{ccc}
0&0&C_{13}\\
0&0&0\\
0&0&0
\end{array}\right)+\xi\otimes f\mapsto
\left(\begin{array}{ccc}
0&0&\tilde{C}_{13}\\
0&0&0\\
0&0&0
\end{array}\right)+\xi\otimes \tilde{f},
$$
where $f,\tilde{f}\in(I_{-}^{\mathcal{N}}\vee P_{\xi})^{\perp}$.

When $\psi^{0}$ is order-preserving, we have
$$
\phi:\left(\begin{array}{ccc}
0&C_{12}&C_{13}\\
0&0&0\\
0&0&0
\end{array}\right)\mapsto
\left(\begin{array}{ccc}
0&\tilde{C}_{12}&\tilde{C}_{13}\\
0&0&0\\
0&0&0
\end{array}\right)+\xi\otimes \tilde{f}, \ {\rm{for \ some \ }} \tilde{f}\in(I_{-}^{\mathcal{N}}\vee P_{\xi})^{\perp}.
$$
Although it is not known who the image of $C\in\mathcal{J}(\mathcal{L},N)$ is, according to the mapping above, when $C$ is traversed, the corresponding image's position in $\tilde{C}_{12}$ will always be traversed.

When $\psi^{0}$ is order-reversing, we have $\psi^{0}(0_{+}^{\mathcal{N}})=I_{-}^{\mathcal{N}}$, and therefore
$$
\phi:\left(\begin{array}{ccc}
0&C_{12}&C_{13}\\
0&0&0\\
0&0&0
\end{array}\right)\in\mathcal{J}(\mathcal{L},0_{+}^{\mathcal{N}})\mapsto
\left(\begin{array}{ccc}
0&0&\tilde{C}_{13}\\
0&0&\tilde{C}_{23}\\
0&0&0
\end{array}\right)+\xi\otimes \tilde{f}\in\mathcal{J}(\mathcal{L},I_{-}^{\mathcal{N}}\vee P_{\xi}).
$$
Similarly, when $C$ is traversed, the corresponding image's position in $\tilde{C}_{23}$ will always be traversed.
\end{remark}

For convenience, let $P_{0}=\xi\otimes f_{0}$ with $f_{0}\in (I_{-}^{\mathcal{N}})^{\perp}$ and $\langle \xi, f_{0} \rangle = 1$. Next, we first consider the structure of $\varphi(P_{0})$. By Lemma \ref{3-cases}, $\varphi(P_{0})$ can take one of three forms. Now we consider $\varphi(P_{0})$ of three cases.

(1) The following lemma shows that $\varphi(P_{0})$ does not satisfy Form (3).

\begin{lemma}\label{not3}
There not exists some $N\in\mathcal{N}^{0}$ such that $\varphi(P_{0})$ satisfies Form (3).
\end{lemma}
\begin{proof}
Suppose there exists $N\in\mathcal{N}^{0}$ such that $\varphi(P_{0})$ is of Form (3). If $\psi^{0}(N)=I_{-}^{\mathcal{N}}$, $\varphi(P_{0})$ cannot be of Form (3). Hence we may assume $\psi^{0}(N)\neq I_{-}^{\mathcal{N}}$. With respect to the decomposition $\mathcal{H}=\psi^{0}(N)\oplus (I_{-}^{\mathcal{N}}\ominus \psi^{0}(N))\oplus (I_{-}^{\mathcal{N}})^{\perp}$, we obtain
$$
\varphi(P_{0})=
\left(\begin{array}{ccc}
1&\tilde{P}_{12}&\tilde{P}_{13}\\
0&\tilde{P}_{22}&\tilde{P}_{23}\\
0&0&0
\end{array}\right).
$$

Firstly, we consider the case where $\psi^{0}$ is order-preserving. Since $\varphi(P_{0})$ is of Form (3) with  respect to $\psi^{0}(N)$, the same holds for $0_{+}^{\mathcal{N}}$. Therefore, we may assume without loss of generality that $N=0_{+}^{\mathcal{N}}< I_{-}^{\mathcal{N}}$.

For any $f\in(I_{-}^{\mathcal{N}}\vee P_{\xi})^{\perp}$, we have $[P_{0},\xi\otimes f]=\xi\otimes f$. Then, by Remark \ref{traverse}, there is some $C=\left(\begin{array}{ccc}
0&0&C_{13}\\
0&0&0\\
0&0&0
\end{array}\right)\in\mathcal{J}(\mathcal{L},0_{+}^{\mathcal{N}})$ and $\tilde{f}\in (I_{-}^{\mathcal{N}}\vee P_{\xi})^{\perp}$ such that $\phi(\xi\otimes f)=C+\xi\otimes \tilde{f}$, thus $[\varphi(P_{0}),C+\xi\otimes \tilde{f}]=C+\xi\otimes \tilde{f}$. This yields
$$0=[\varphi(P_{0}),\xi\otimes \tilde{f}]=\xi\otimes \tilde{f}$$
and hence the mapping $\phi$ (resp. $\psi$) acts as
\begin{align}\label{xi to *}
\xi\otimes f\mapsto C=\left(\begin{array}{ccc}
0&0&C_{13}\\
0&0&0\\
0&0&0
\end{array}\right).
\end{align}
For any
$
C=\left(\begin{array}{ccc}
0&0&C_{13}\\
0&0&0\\
0&0&0
\end{array}\right)\in \mathcal{J}(\mathcal{L},0_{+}^{\mathcal{N}}),
$
we have $[P_{0},C]=0$, which implies $[\varphi(P_{0}),\phi(C)]=0$. By Remark \ref{traverse}, we obtain
$$
\phi(C)=
\left(\begin{array}{ccc}
0&0&\tilde{C}_{13}\\
0&0&0\\
0&0&0
\end{array}\right)+\xi\otimes f',  \ {\rm{for \ some}} \ f'\in (I_{-}^{\mathcal{N}}\vee P_{\xi})^{\perp},
$$
which we denote as $\tilde{C} + \xi \otimes f'$. Since $[\varphi(P_{0}),\phi(C)]=0$, we have $[\varphi(P_{0}),\tilde{C}]=0$, and hence $\tilde{C}=0$. Therefore, the mapping $\phi$ acts as
\begin{align}\label{* to xi}
\left(\begin{array}{ccc}
0&0&C_{13}\\
0&0&0\\
0&0&0
\end{array}\right)\in\mathcal{J}(\mathcal{L},0_{+}^{\mathcal{N}})\mapsto\xi\otimes f', \ {\rm{for \ some \ }}f'\in (I_{-}^{\mathcal{N}}\vee P_{\xi})^{\perp}.
\end{align}

For any $D_{1}\in\mathcal{J}(\mathcal{L},0_{+}^{\mathcal{N}}\vee P_{\xi})$, there exists $\bar{D}_{1}\in \mathcal{J}(\mathcal{L},0_{+}^{\mathcal{N}})$ and $f_{1}\in (I_{-}^{\mathcal{N}}\vee P_{\xi})^{\perp}$ such that $D_{1}=\bar{D}_{1}+\xi\otimes f_{1}$, and we have $[P_{0},D_{1}]=\xi\otimes f_{1}$. This implies $[\varphi(P_{0}),\tilde{D}_{1}+\xi\otimes \tilde{f}_{1}]=\psi(\xi\otimes f_{1})$ for all $\tilde{D}_{1}\in \mathcal{J}(\mathcal{L},0_{+}^{\mathcal{N}})$ and $\tilde{f}_{1}\in (I_{-}^{\mathcal{N}}\vee P_{\xi})^{\perp}$, where
$$
\tilde{D}_{1}=\left(\begin{array}{ccc}
0&\tilde{D}_{12}&\tilde{D}_{13}\\
0&0&0\\
0&0&0
\end{array}\right) \ {\rm{and}} \
\psi(\xi\otimes f_{1})=\phi(\xi\otimes f_{1})=\left(\begin{array}{ccc}
0&0&\ast\\
0&0&0\\
0&0&0
\end{array}\right).
$$
It follows that $[\varphi(P_{0}),\tilde{D}_{1}]=\psi(\xi\otimes f_{1})$, and hence $\tilde{D}_{12}-\tilde{D}_{12}\tilde{P}_{22}=0$ for all $\tilde{D}_{12}\in\mathcal{B}((I_{-}^{\mathcal{N}}\ominus 0_{+}^{\mathcal{N}}),0_{+}^{\mathcal{N}})$, which implies $\tilde{P}_{22}=1$. Since $\varphi({P_{0}})$ is an idempotent, we conclude that
$$
\varphi(P_{0})=
\left(\begin{array}{ccc}
1&0&\tilde{P}_{13}\\
0&1&\tilde{P}_{23}\\
0&0&0
\end{array}\right).
$$
For any $E_{1}\in\mathcal{J}(\mathcal{L},I_{-}^{\mathcal{N}}\vee P_{\xi})$, there exists some $\bar{E}_{1}\in \mathcal{J}(\mathcal{L},I_{-}^{\mathcal{N}})$ and $g_{1}\in (I_{-}^{\mathcal{N}}\vee P_{\xi})^{\perp}$ such that $E_{1}=\bar{E}_{1}+\xi\otimes g_{1}$, and we have $[P_{0},E_{1}]=\xi\otimes g_{1}$. This implies $[\varphi(P_{0}),\tilde{E}_{1}+\xi\otimes \tilde{g}_{1}]=\psi(\xi\otimes g_{1})$ for all $\tilde{E}_{1}\in\mathcal{J}(\mathcal{L},I_{-}^{\mathcal{N}})$ and $\tilde{g}_{1}\in (I_{-}^{\mathcal{N}}\vee P_{\xi})^{\perp}$, where
$$
\tilde{E}_{1}=\left(\begin{array}{ccc}
0&0&\tilde{E}_{13}\\
0&0&\tilde{E}_{23}\\
0&0&0
\end{array}\right) \ {\rm{and}} \
\psi(\xi\otimes g_{1})=\left(\begin{array}{ccc}
0&0&*\\
0&0&0\\
0&0&0
\end{array}\right).
$$
It follows that $[\varphi(P_{0}),\tilde{E}_{1}]=\psi(\xi\otimes g_{1})$, and hence $\tilde{E}_{23}=0$, which leads to a contradiction.

Secondly, we consider the case where $\psi^{0}$ is order-reversing. Similarly, we can assume that $0_{+}^{\mathcal{N}}<N=I_{-}^{\mathcal{N}}$, and we can obtain the results (\ref{xi to *}) and (\ref{* to xi}).

For any $D_{2}\in\mathcal{J}(\mathcal{L},I_{-}^{\mathcal{N}}\vee P_{\xi})$, there exists some $\bar{D}_{2}\in \mathcal{J}(\mathcal{L},I_{-}^{\mathcal{N}})$ and $f_{2}\in (I_{-}^{\mathcal{N}}\vee P_{\xi})^{\perp}$ such that $D_{2}=\bar{D}_{2}+\xi\otimes f_{2}$, and we have $[P_{0},D_{2}]=\xi\otimes f_{2}$. This implies $[\varphi(P_{0}),\tilde{D}_{2}+\xi\otimes \tilde{f}_{2}]=\psi(\xi\otimes f_{2})$ for all $\tilde{D}_{2}\in\mathcal{J}(\mathcal{L},0_{+}^{\mathcal{N}})$ and $\tilde{f}_{2}\in (I_{-}^{\mathcal{N}}\vee P_{\xi})^{\perp}$, where
$$
\tilde{D}_{2}=\left(\begin{array}{ccc}
0&\tilde{D}_{12}&\tilde{D}_{13}\\
0&0&0\\
0&0&0
\end{array}\right) \ {\rm{and}} \
\psi(\xi\otimes f_{2})=\left(\begin{array}{ccc}
0&0&\ast\\
0&0&0\\
0&0&0
\end{array}\right).
$$
It follows that $[\varphi(P_{0}),\tilde{D}_{2}]=\psi(\xi\otimes f_{2})$, and hence $\tilde{D}_{12}-\tilde{D}_{12}\tilde{P}_{22}=0$ for all $\tilde{D}_{12}\in\mathcal{B}((I_{-}^{\mathcal{N}}\ominus 0_{+}^{\mathcal{N}}),0_{+}^{\mathcal{N}})$, which implies $\tilde{P}_{22}=1$. Since $\varphi({P_{0}})$ is an idempotent, we conclude that
$$
\varphi(P_{0})=
\left(\begin{array}{ccc}
1&0&\tilde{P}_{13}\\
0&1&\tilde{P}_{23}\\
0&0&0
\end{array}\right).
$$

For any $E_{2}\in\mathcal{J}(\mathcal{L},0_{+}^{\mathcal{N}}\vee P_{\xi})$, there exists some $\bar{E}_{2}\in \mathcal{J}(\mathcal{L},0_{+}^{\mathcal{N}})$ and $g_{2}\in (I_{-}^{\mathcal{N}}\vee P_{\xi})^{\perp}$ such that $E_{2}=\bar{E}_{2}+\xi\otimes g_{2}$, and we have $[P_{0},E_{2}]=\xi\otimes g_{2}$. This implies $[\varphi(P_{0}),\tilde{E}_{2}+\xi\otimes \tilde{g}_{2}]=\psi(\xi\otimes g_{2})$ for all $\tilde{E}_{2}\in\mathcal{J}(\mathcal{L},I_{-}^{\mathcal{N}})$ and $\tilde{g}_{2}\in (I_{-}^{\mathcal{N}}\vee P_{\xi})^{\perp}$, where
$$
\tilde{E}_{2}=\left(\begin{array}{ccc}
0&0&\tilde{E}_{13}\\
0&0&\tilde{E}_{23}\\
0&0&0
\end{array}\right) \ {\rm{and}} \
\psi(\xi\otimes g_{2})=\left(\begin{array}{ccc}
0&0&\ast\\
0&0&0\\
0&0&0
\end{array}\right).
$$
It follows that $[\varphi(P_{0}),\tilde{E}_{2}]=\psi(\xi\otimes g_{2})$, and hence $\tilde{E}_{23}=0$, which leads to a contradiction.
\end{proof}

(2) The following lemma shows that $\varphi(P_{0})$ does not satisfy Form (2).
\begin{lemma}\label{not2}
There not exists some $N\in\mathcal{N}^{0}$ such that $\varphi(P_{0})$ satisfies Form (2).
\end{lemma}
\begin{proof}
If there exists some $N\in\mathcal{N}^{0}$ such that $\varphi(P_{0})$ satisfies Form (2), by the decomposition $\mathcal{H}=(\psi^{0}(N)\vee P_{\xi})\oplus (\psi^{0}(N)\vee P_{\xi})^{\perp}$, we have
$$
\varphi(P_{0})=
\left(\begin{array}{ccc}
1&*\\
0&*(\neq 0,1)\\
\end{array}\right).
$$
Indeed, if there exists some $N\in\mathcal{N}^{0}$ satisfying the aforementioned conditions, then $0_{+}^{\mathcal{N}}$ must also satisfy these conditions. Therefore, we may assume without loss of generality that $\psi^{0}(N)=0_{+}^{\mathcal{N}}$. And under the decomposition $\mathcal{H}=0_{+}^{\mathcal{N}}\oplus (I_{-}^{\mathcal{N}}\ominus 0_{+}^{\mathcal{N}})\oplus (I_{-}^{\mathcal{N}})^{\perp}$, we can write $\xi=(\xi_{1},\xi_{2},\xi_{3})^{T}$ and
$$
\varphi(P_{0})=
\left(\begin{array}{ccc}
1&\tilde{P}_{12}&\tilde{P}_{13}\\
0&\tilde{P}_{22}&\tilde{P}_{23}\\
0&0&\tilde{P}_{33}
\end{array}\right) \ {\rm{with}} \ \tilde{P}_{33}\xi_{3}=\xi_{3}.
$$

Now, we consider the case where $\psi^{0}$ is order-preserving.

For any
$
C=\left(\begin{array}{ccc}
0&C_{12}&C_{13}\\
0&0&0\\
0&0&0
\end{array}\right)\in\mathcal{J}(\mathcal{L},0_{+}^{\mathcal{N}})\subset \mathcal{J}(\mathcal{L},0_{+}^{\mathcal{N}}\vee P_{\xi}),
$
we have $[P_{0},C]=0$, which implies $[\varphi(P_{0}),\phi(C)]=0$. Then, by Remark \ref{traverse}, we obtain
$$
\phi(C)=
\left(\begin{array}{ccc}
0&\tilde{C}_{12}&\tilde{C}_{13}\\
0&0&0\\
0&0&0
\end{array}\right)+\xi\otimes f,  \ {\rm{for \ some}} \ f\in (I_{-}^{\mathcal{N}}\vee P_{\xi})^{\perp},
$$
which we denote as $\tilde{C} + \xi \otimes f$.

Hence $[\varphi(P_{0}),\tilde{C}]=0$, and $\tilde{C}_{12}-\tilde{C}_{12}\tilde{P}_{22}=0$ for all $\tilde{C}_{12}\in\mathcal{B}((I_{-}^{\mathcal{N}}\ominus 0_{+}^{\mathcal{N}}),0_{+}^{\mathcal{N}})$, it follows that $\tilde{P}_{22}=1$. Since $\varphi(P_{0})$ is an idempotent, we have
$$
\varphi(P_{0})=
\left(\begin{array}{ccc}
1&0&\tilde{P}_{13}\\
0&1&\tilde{P}_{23}\\
0&0&\tilde{P}_{33}
\end{array}\right).
$$
For any
$
D=\left(\begin{array}{ccc}
0&0&D_{13}\\
0&0&D_{23}\\
0&0&0
\end{array}\right)\in\mathcal{J}(\mathcal{L},I_{-}^{\mathcal{N}})\subset \mathcal{J}(\mathcal{L},I_{-}^{\mathcal{N}}\vee P_{\xi}),
$
we have $[P_{0},D]=0$, which implies $[\varphi(P_{0}),\phi(D)]=0$. And, by Remark \ref{traverse}, we obtain
$$
\phi(D)=
\left(\begin{array}{ccc}
0&0&\tilde{D}_{13}\\
0&0&\tilde{D}_{23}\\
0&0&0
\end{array}\right)+\xi\otimes f',  \ {\rm{for \ some}} \ f'\in (I_{-}^{\mathcal{N}}\vee P_{\xi})^{\perp},
$$
which we denote as $\tilde{D} + \xi \otimes f'$.

Hence $[\varphi(P_{0}),\tilde{D}]=0$, and $\tilde{D}_{23}-\tilde{D}_{23}\tilde{P}_{33}=0$ for all $\tilde{D}_{23}\in\mathcal{B}((I_{-}^{\mathcal{N}})^{\perp},(I_{-}^{\mathcal{N}}\ominus 0_{+}^{\mathcal{N}}))$ with $\tilde{D}_{23}\xi_{3}=0$, and hence $\tilde{P}_{33}^{\ast}(I_{-}^{\mathcal{N}}\vee P_{\xi})^{\perp}=1$. Since $\tilde{P}_{33}\xi_{3}=\xi_{3}$ and $\tilde{P}_{33}$ is an idempotent, we obtain $\tilde{P}_{33}=1$. 
Again, by the idempotence of $\varphi(P_{0})$, we obtain $\varphi(P_{0})=1$, which leads to a contradiction.

Since $[P_{0},C]=0=[P_{0},D]$ for any $C\in\mathcal{J}(\mathcal{L},0_{+}^{\mathcal{N}})$ and $D\in\mathcal{J}(\mathcal{L},I_{-}^{\mathcal{N}})$, the situation remains the same regardless of whether $\psi^0$ is order-preserving or order-reversing. A detailed proof can be found in the proof of Lemma \ref{not3}.\end{proof}

\begin{proposition}\label{xi-f00}
Let $P_{0}=\xi\otimes f_{0}$ for any $f_{0}\in (I_{-}^{\mathcal{N}})^{\perp}$ with $\langle\xi, f_{0}\rangle=1$. Then there exists $\tilde{f}_{0}\in (I_{-}^{\mathcal{N}})^{\perp}$ with $\langle\xi, \tilde{f}_{0}\rangle=1$ such that $\varphi(P_{0})=\xi\otimes \tilde{f}_{0}$.
\end{proposition}
\begin{proof}
By Lemma \ref{not3} and Lemma \ref{not2}, we obtain $\varphi(P_{0})$ satisfies Form (1) for all $N\in\mathcal{N}^{0}$. Let $N=0_{+}^{\mathcal{N}}$. 
Then, under the decomposition $\mathcal{H}= 0_{+}^{\mathcal{N}} \oplus (I_{-}^{\mathcal{N}} \ominus 0_{+}^{\mathcal{N}}) \oplus (I_{-}^{\mathcal{N}})^{\perp}$, we have $\varphi(P_{0})=\tilde{P}_{0}+\xi\otimes \tilde{f}_{0}$ with $\tilde{f}_{0}\in (I_{-}^{\mathcal{N}})^{\perp}$ and
$$
\tilde{P}_{0}=\left(\begin{array}{ccc}
\tilde{P}_{11}&\tilde{P}_{12}&\tilde{P}_{33}\\
0&0&0\\
0&0&0
\end{array}\right).
$$

Now, we consider the case where $\psi^{0}$ is order-preserving.

Since $[P_{0},C]=0$ for all $C\in\mathcal{J}(\mathcal{L},0_{+}^{\mathcal{N}})$, it follows that $[\varphi(P_{0}),\tilde{C}]=0$ for the corresponding operator $\tilde{C}=\tilde{C}_{0}+\xi\otimes g$ with $g\in (I_{-}^{\mathcal{N}}\vee P_{\xi})^{\perp}$ and
$$
\tilde{C}_{0}=\left(\begin{array}{ccc}
0&\tilde{C}_{12}&\tilde{C}_{13}\\
0&0&0\\
0&0&0
\end{array}\right).
$$
This implies $[\tilde{P}_{0},\tilde{C}_{0}]=0$, and so $\tilde{P}_{11}\tilde{C}_{12}=0$ for all $\tilde{C}_{12}\in\mathcal{B}(I_{-}^{\mathcal{N}} \ominus 0_{+}^{\mathcal{N}},0_{+}^{\mathcal{N}})$. Thus $\tilde{P}_{11}=0$, and hence $\tilde{P}_{0}=0$. Therefore, $\varphi(P_{0})=\xi\otimes \tilde{f}_{0}$. Since $\varphi(P)$ is an idempotent, we have $\langle\xi, \tilde{f}_{0}\rangle=1$.

When $\psi^{0}$ is order-reversing, we consider $[P_{0},C]$ for any $C\in\mathcal{J}(\mathcal{L},I_{-}^{\mathcal{N}})$ and obtain the same result.
\end{proof}
\begin{remark}\label{xi-eta1}
In this remark, we only consider the case where $\psi^{0}$ is order-preserving. The proof for the order-reversing case is analogous.

(1) Consider the decomposition $\mathcal{H}= 0_{+}^{\mathcal{N}} \oplus (I_{-}^{\mathcal{N}} \ominus 0_{+}^{\mathcal{N}}) \oplus (I_{-}^{\mathcal{N}})^{\perp}$. For any matrix of the form
$$
C=\left(\begin{array}{ccc}
0&0&C_{13}\\
0&0&0\\
0&0&0
\end{array}\right)\in\mathcal{J}(\mathcal{L},0_{+}^{\mathcal{N}}),
$$
we have $[\xi\otimes f_{0},C]=0$. It follows from Proposition \ref{xi-f00} that $[\xi\otimes \tilde{f}_{0},\phi(C)]=0$ for some $\tilde{f}_{0}\in (I_{-}^{\mathcal{N}})^{\perp}$ with $\langle\xi, \tilde{f}_{0}\rangle=1$. By Remark \ref{xi-eta} or Remark \ref{traverse}, it follows that
$$
\phi:\left(\begin{array}{ccc}
0&0&C_{13}\\
0&0&0\\
0&0&0
\end{array}\right)\mapsto
\left(\begin{array}{ccc}
0&0&\tilde{C}_{13}\\
0&0&0\\
0&0&0
\end{array}\right).
$$
(2) For any $f\in(I_{-}^{\mathcal{N}}\vee P_{\xi})^{\perp}$, we have $[\xi\otimes f_{0},\xi\otimes f]=\xi\otimes f$. Similarly, $\phi$ maps $\xi\otimes f$ to $\xi\otimes \tilde{f}$, where $f,\tilde{f}\in(I_{-}^{\mathcal{N}}\vee P_{\xi})^{\perp}$ and $\tilde{f}$ corresponds to $f$.

\end{remark}

\begin{corollary}\label{J(L,N)2}
Let $N\in\mathcal{N}^{0}$ and $C\in\mathcal{J}(\mathcal{L},N)$. Then $\phi(C)\in\mathcal{J}(\mathcal{L},\psi^{0}(N))$. Moreover, $\phi(\mathcal{J}(\mathcal{L},N))=\mathcal{J}(\mathcal{L},\psi^{0}(N))$
\end{corollary}
\begin{proof}
The proof is similar to that of Corollary \ref{J(L,N)}, so we omit it.
\end{proof}

\begin{proposition}\label{Omega2}
For any $Q \in \Omega_{2}(\mathcal{L}, N)$, we have $\varphi(Q)\in \Omega_{2}(\mathcal{L}, \psi^{0}(N))$.
\end{proposition}
\begin{proof}
The proof considers two cases: whether $\psi^{0}$ is order-reversing or order-preserving. As the reasoning for both cases is analogous to that of Proposition \ref{Omega1}, we omit the details.
\end{proof}

\begin{corollary}\label{order-preserving}
Let $0_{+}^{\mathcal{N}}\neq 0$. Then $\psi^{0}$ is order-preserving.
\end{corollary}
\begin{proof}
Suppose $\psi^{0}$ is order-reversing, by Proposition \ref{Omega2}, under the decomposition $\mathcal{H}=0_{+}^{\mathcal{N}}\oplus(I_{-}^{\mathcal{N}}\ominus 0_{+}^{\mathcal{N}})\oplus (I_{-}^{\mathcal{N}})^{\perp}$, the mapping $\varphi$ acts as
$$
\left(\begin{array}{ccc}
Q_{11}&Q_{12}&Q_{13}\\
0&0&0\\
0&0&0
\end{array}\right)
\mapsto\left(\begin{array}{ccc}
\tilde{Q}_{11}&\tilde{Q}_{12}&\tilde{Q}_{13}\\
0&\tilde{Q}_{22}&\tilde{Q}_{23}\\
0&0&0
\end{array}\right)
$$
and
$$
\left(\begin{array}{ccc}
P_{11}&P_{12}&P_{13}\\
0&P_{22}&P_{23}\\
0&0&0
\end{array}\right)
\mapsto\left(\begin{array}{ccc}
\tilde{P}_{11}&\tilde{P}_{12}&\tilde{P}_{13}\\
0&0&0\\
0&0&0
\end{array}\right).
$$
Since $\Omega_{2}(\mathcal{L},0_{+}^{\mathcal{N}})\subset\Omega_{2}(\mathcal{L}, I_{-}^{\mathcal{N}})$, and consider $\varphi^{-1}$, then yields a contradiction.
\end{proof}

Similar to Corollary \ref{=J(L,E)}, we obtain the following corollary.
\begin{corollary}
Let $N\in\mathcal{N}^{0}$. Then $\psi(C)=\phi(C)$ for all $C\in\mathcal{J}(\mathcal{L},N)$, and $\psi(\xi\otimes f)=\phi(\xi\otimes f)$ for all $f\in (I^{\mathcal{N}}_{-})^{\perp}$ with $\langle \xi,f\rangle=0$, i.e., $\psi(D)=\phi(D)$ for all $D\in \mathcal{J}(\mathcal{L},N\vee P_{\xi})$.
\end{corollary}

\section{Main Results}
Let $0<N<I_{-}^{\mathcal{N}}$. We represent operators by $3\times 3$ matrices (resp. $2\times 2$ matrices) relative to the decomposition
$$\mathcal{H}=N\oplus (I_{-}^{\mathcal{N}}\ominus N)\oplus (I_{-}^{\mathcal{N}})^{\perp} \ ({\rm{resp.}} \ \mathcal{H}=N\oplus N^{\perp}).$$
The images of the operators under the mappings $\psi,\phi,\varphi$ are expressed with respect to the decomposition
$$\mathcal{H}=\psi^{0}(N)\oplus (I_{-}^{\mathcal{N}}\ominus \psi^{0}(N))\oplus (I_{-}^{\mathcal{N}})^{\perp} \ ({\rm{resp.}} \ \mathcal{H}=\psi^{0}(N)\oplus \psi^{0}(N)^{\perp}).$$
In the rest of this proof, we always work under the above decomposition without further mention.

\begin{theorem}
Let $\psi$ be a Lie automorphism on ${\rm{Alg}}\mathcal{L}$. Then there are an automorphism $\epsilon$ and a linear functional $\tau$ on ${\rm{Alg}}\mathcal{L}$ vanishing on each commutator such that
$$
\psi(A)=\epsilon(A)+\tau(A)I, \ \forall A\in{\rm{Alg}}\mathcal{L}.
$$
\end{theorem}
\begin{proof}
The proof of the theorem proceeds in six steps.

Step 1: For any operator
$$
A=\left(\begin{array}{ccc}
A_{11}&A_{12}&A_{13}\\
0&A_{22}&A_{23}\\
0&0&0
\end{array}\right)\in{\rm{Alg}}\mathcal{L},
$$
we have $[A,\xi\otimes f]=0$ for all $f\in (I_{-}^{\mathcal{N}})^{\perp}$. Combining Remark \ref{xi-eta}, Lemma \ref{xi-f0}, Proposition \ref{xi-f00} and Remark \ref{xi-eta1}, we obtain $[\psi(A),\xi\otimes f]=0$ for any $f\in (I_{-}^{\mathcal{N}})^{\perp}$. Thus, there is some $\lambda\in\mathbf{C}$ such that
\begin{align}\label{11}
\psi(A)=\left(\begin{array}{ccc}
\tilde{A}_{11}&\tilde{A}_{12}&\tilde{A}_{13}\\
0&\tilde{A}_{22}&\tilde{A}_{23}\\
0&0&0
\end{array}\right)+\lambda I.
\end{align}
Take $A\in{\rm{Alg}}\mathcal{L}$ to be of the form
$
\left(\begin{array}{ccc}
0&0&0\\
0&A_{22}&A_{23}\\
0&0&0
\end{array}\right).
$
For any
$$
C=\left(\begin{array}{ccc}
0&0&C_{13}\\
0&0&0\\
0&0&0
\end{array}\right)\in\mathcal{J}(\mathcal{L},N),
$$
there exist scalars $\lambda,\mu\in\mathbb{C}$ such that
$$
\psi(A)=\left(\begin{array}{ccc}
\tilde{A}_{11}&\tilde{A}_{12}&\tilde{A}_{13}\\
0&\tilde{A}_{22}&\tilde{A}_{23}\\
0&0&0
\end{array}\right)+\lambda I \ {\rm{and}} \
\psi(C)=
\left(\begin{array}{ccc}
0&0&\tilde{C}_{13}\\
0&0&0\\
0&0&0
\end{array}\right)+\mu I.
$$

From the Lie product $[A,C]$, we consider $[\psi(A),\psi(C)]$. This yields $\tilde{A}_{11}\tilde{C}_{13}=0$, 
and hence $\tilde{A}_{11}=0$. Consequently, the mapping $\psi$ acts as
\begin{align}\label{12}
\left(\begin{array}{ccc}
0&0&0\\
0&A_{22}&A_{23}\\
0&0&0
\end{array}\right)\mapsto
\left(\begin{array}{ccc}
0&\tilde{A}_{12}&\tilde{A}_{13}\\
0&\tilde{A}_{22}&\tilde{A}_{23}\\
0&0&0
\end{array}\right)+\lambda I, \ {\rm{for \ some}} \ \lambda\in\mathbb{C}.
\end{align}


For any
$
B=\left(\begin{array}{ccc}
0&0&B_{13}\\
0&0&B_{23}\\
0&0&B_{33}
\end{array}\right)\in{\rm{Alg}}\mathcal{L} \ {\rm{and \ let \ }}
\psi(B)=\left(\begin{array}{ccc}
\tilde{B}_{11}&\tilde{B}_{12}&\tilde{B}_{13}\\
0&\tilde{B}_{22}&\tilde{B}_{23}\\
0&0&\tilde{B}_{33}
\end{array}\right).
$
Consider $[B,P]$ for any
$
P=\left(\begin{array}{ccc}1&P_{12}&P_{13}\\0&0&0\\0&0&0\end{array}\right)
$,
by Propositions \ref{Omega1}, \ref{Omega2}, we assume that
$
\varphi(P)=\left(\begin{array}{ccc}1&\tilde{P}_{12}&\tilde{P}_{13}\\0&0&0\\0&0&0\end{array}\right).
$
Thus, by corollaries \ref{J(L,N)}, \ref{J(L,N)2}, we obtain $\tilde{B}_{11}\tilde{P}_{12}-\tilde{B}_{12}-\tilde{P}_{12}\tilde{B}_{22}=0$. Which implies $\tilde{B}_{12}=0$, and there is some scalar $\lambda\in\mathbb{C}$ such that $\tilde{B}_{11}=\lambda$ and $\tilde{B}_{22}=\lambda$. Therefore, the mapping $\psi$ acts as
\begin{align}\label{13}
\left(\begin{array}{ccc}
0&0&B_{13}\\
0&0&B_{23}\\
0&0&B_{33}
\end{array}\right)\mapsto
\left(\begin{array}{ccc}
0&0&\tilde{B}_{13}\\
0&0&\tilde{B}_{23}\\
0&0&\tilde{B}_{33}
\end{array}\right)+\lambda I, \ {\rm{for \ some}} \ \lambda\in\mathbb{C}.
\end{align}

Now take $A$ to be of the form
$
\left(\begin{array}{ccc}
A_{11}&A_{12}&A_{13}\\
0&0&0\\
0&0&0
\end{array}\right)\in{\rm{Alg}}\mathcal{L}.
$
For any
$
B=\left(\begin{array}{ccc}
0&0&B_{13}\\
0&0&B_{23}\\
0&0&B_{33}
\end{array}\right)\in{\rm{Alg}}\mathcal{L},
$
by Eq. (\ref{11}) and Eq. (\ref{13}), there exist two scalars $\lambda,\mu\in\mathbb{C}$ such that
$$
\psi(A)=\left(\begin{array}{ccc}
\tilde{A}_{11}&\tilde{A}_{12}&\tilde{A}_{13}\\
0&\tilde{A}_{22}&\tilde{A}_{23}\\
0&0&0
\end{array}\right)+\lambda I \ {\rm{and}} \
\psi(B)=
\left(\begin{array}{ccc}
0&0&\tilde{B}_{13}\\
0&0&\tilde{B}_{23}\\
0&0&\tilde{B}_{33}
\end{array}\right)+\mu I.
$$
From the Lie product $[A,B]$, we consider $[\psi(A),\psi(B)]$. This yields
$$
\tilde{A}_{22}\tilde{B}_{23}+\tilde{A}_{23}\tilde{B}_{33}=0.
$$
For any $x\in I_{-}^{\mathcal{N}}\ominus N$, $y\in(I_{-}^{\mathcal{N}})^{\perp}$, and nonzero element $f\in (I_{-}^{\mathcal{N}}\vee P_{\xi})^{\perp}$, let us choose $\tilde{B}_{23}=x\otimes f$ and $\tilde{B}_{33}=y\otimes f$. Substituting into the above equation gives $\tilde{A}_{22}=0$ and $\tilde{A}_{23}=0$. Consequently, the mapping $\psi$ acts as
\begin{align}\label{14}
\left(\begin{array}{ccc}
A_{11}&A_{12}&A_{13}\\
0&0&0\\
0&0&0
\end{array}\right)\to
\left(\begin{array}{ccc}
\tilde{A}_{11}&\tilde{A}_{12}&\tilde{A}_{13}\\
0&0&0\\
0&0&0
\end{array}\right)+\lambda I, \ {\rm{for \ some}} \ \lambda\in\mathbb{C}.
\end{align}

Step 2: With the decomposition $\mathcal{H}=N\oplus N^{\perp}$, we can express $\xi$ as $\xi = \xi_{N} \oplus \xi_{N^{\perp}}$. Define $e_{N}=\frac{1}{\|\xi_{N^{\perp}}\|^2} \xi_{N}\otimes\xi_{N^{\perp}}$ and
$$
E_{N}=\left(\begin{array}{cc}
1&-e_{N}\\
0&0\\
\end{array}\right).
$$
By Propositions \ref{Omega1}, \ref{Omega2}, there exists an operator $\tilde{e}_{N}\in\mathcal{B}(\psi^{0}(N)^{\perp},\psi^{0}(N))$ such that $\psi(E_{N})-\tilde{E}_{N}\in\mathbb{C}I$, where
$$
\tilde{E}_{N}=\left(\begin{array}{cc}
1&-\tilde{e}_{N}\\
0&0\\
\end{array}\right)\in{\rm{Alg}}\mathcal{L}.
$$
Now, consider the algebra $\mathcal{A}_{N}$ generated by operators of the form
$\left(\begin{array}{cc}
A_{N}&-A_{N}e_{N}\\
0&0\\
\end{array}\right)
$
in ${\rm{Alg}}\mathcal{L}$.
For any $A\in\mathcal{A}_{N}$, by Eq. (\ref{14}), there is an operator $\tilde{A}\in{\rm{Alg}}\mathcal{L}$ such that $\psi(A)-\tilde{A}\in\mathbb{C}I$, where
$$
\tilde{A}=\left(\begin{array}{cc}
\tilde{A}_{N}&X\\
0&0\\
\end{array}\right).
$$
Since $[A,E_{N}]=0$, it follows that $[\psi(A),\tilde{E}_{N}]=0$, which implies $X=-\tilde{A}_{N}\tilde{E}_{N}$. Therefore, the mapping $\psi$ satisfies
\begin{align}\label{3.1}
\left(\begin{array}{cc}
A_{N}&-A_{N}e_{N}\\
0&0\\
\end{array}\right)\mapsto
\left(\begin{array}{cc}
\tilde{A}_{N}&-\tilde{A}_{N}\tilde{e}_{N}\\
0&0\\
\end{array}\right)+\mathbb{C}I.
\end{align}

Let $\tilde{\mathcal{A}}_{N}$ be the algebra generated by operator of the form
$
\left(\begin{array}{cc}
\tilde{A}_{N}&-\tilde{A}_{N}\tilde{e}_{N}\\
0&0\\
\end{array}\right)
$
in ${\rm{Alg}}\mathcal{L}$.
Then, we can define the linear bijective mapping $\epsilon_{N}:\mathcal{A}_{N}\to\tilde{\mathcal{A}}_{N}$ with $\psi(A)-\epsilon_{N}(A)\in\mathbb{C}I$ for any $A\in\mathcal{A}_{N}$. Indeed, since $\psi$ is a linear mapping, we have $\epsilon_{N}$ is a linear mapping.  For injectivity, suppose $A\in\mathcal{A}_{N}$ satisfies $\epsilon_{N}(A)=0$ , then $\psi(A)\in\mathbb{C}I$, this implies $A\in\mathbb{C}I$, that is say $A=0$. Next we prove the surjective. We may apply the same argument to $\psi^{-1}$. Given any
$$Y=\left(\begin{array}{cc}
\tilde{A}_{N}&-\tilde{A}_{N}\tilde{e}_{N}\\
0&0
\end{array}\right)\in\tilde{\mathcal{A}}_{N},
$$
there exists
$$
X=\left(\begin{array}{cc}
A_{N}&Z\\
0&0
\end{array}\right)\in{\rm{Alg}}\mathcal{L}
$$
and a scalar $\lambda\in\mathbb{C}$ such that $\psi^{-1}(Y)=X+\lambda I$. Since $\psi$ is linear and $\psi(\mathbb{C}I)=\mathbb{C}I$, it follows that $Z=-A_{N}e_{N}$. Then $X\in\mathcal{A}_{N}$, and by the definition of $\epsilon_{N}$, we have $\epsilon_{N}(X)=Y$.

Step 3: The result of Step 2 remains valid for $N=I_{-}^{\mathcal{N}}$. We now turn to proving that $\epsilon_{I_{-}^{\mathcal{N}}}$ is an automorphism. For simplicity, we denote it by $\epsilon_{1}$, and set $\mathcal{A}_{1}=\mathcal{A}_{I_{-}^{\mathcal{N}}}$, $\tilde{\mathcal{A}}_{1}=\tilde{\mathcal{A}}_{I_{-}^{\mathcal{N}}}$, $e=e_{I_{-}^{\mathcal{N}}}$, $\tilde{e}=\tilde{e}_{I_{-}^{\mathcal{N}}}$, $E_{1}=E_{I_{-}^{\mathcal{N}}}$ and $\tilde{E}_{1}=\tilde{E}_{I_{-}^{\mathcal{N}}}$.

Let $\mathcal{N}_{1}=\{N\in\mathcal{N}:N\leq I_{-}^{\mathcal{N}}\}$. Consider the mapping
$$
\mathcal{T}(\mathcal{N}_{1})\to\mathcal{A}_{1}\to\tilde{\mathcal{A}}_{1}\to\mathcal{T}(\mathcal{N}_{1}),
$$
where
$$
\left(\begin{array}{cc}
A_{1}&0\\
0&0\\
\end{array}\right)
\mapsto\left(\begin{array}{cc}
A_{1}&-A_{1}e\\
0&0\\
\end{array}\right)\mapsto
\left(\begin{array}{cc}
\tilde{A}_{1}&-\tilde{A}_{1}\tilde{e}\\
0&0\\
\end{array}\right)\mapsto
\left(\begin{array}{cc}
\tilde{A}_{1}&0\\
0&0\\
\end{array}\right).
$$
According to \cite{Marcoux-Sourour}, we have $\epsilon_{1}=\phi_{1}+\tau_{1}$, where $\phi_{1}$ is an isomorphism and $\tau_{1}$ is a linear central-value mapping that vanishes on every commutator.

Let
$
\phi_{1}(E_{1})=\left(\begin{array}{cc}
\bar{E}_{11}&-\bar{E}_{11}\tilde{e}\\
0&0
\end{array}\right) \ {\rm{and}} \
\tau_{1}(E_{1})=\left(\begin{array}{cc}
\lambda&-\lambda\tilde{e}\\
0&0
\end{array}\right) .
$
From $\epsilon_{1}(E_{1})=\phi_{1}(E_{1})+\tau_{1}(E_{1})$, we obtain $1=\bar{E}_{11}+\lambda$. Since $\bar{E}_{11}$ is a nonzero idempotent, it follows that $\lambda=0$, and hence $\tau_{1}(E_{1})=0$.

Next, we prove $\tau_{1}=0$. For any idempotent
$$
A=\left(\begin{array}{cc}
A_{11}&-A_{11}e\\
0&0
\end{array}\right)
$$
with $A_{11}\neq 0,1$, let
$$
\epsilon_{1}(A)=\left(\begin{array}{cc}
\tilde{A}_{11}&-\tilde{A}_{11}\tilde{e}\\
0&0
\end{array}\right), \
\tau_{1}(A)=\left(\begin{array}{cc}
\mu&-\mu\tilde{e}\\
0&0
\end{array}\right) .
$$
Then $\tilde{A}_{11}-\mu$ is an idempotent, this forces $\mu=0$. Hence, $\tau(A)=0$ for any idempotent $A$ in $\mathcal{A}_{1}$.

Next, we prove that $\tau(A)=0$ for any $A\in\mathcal{A}_{1}$. Indeed, under the decomposition $\mathcal{H}=0_{+}^{\mathcal{N}}\oplus (I_{-}^{\mathcal{N}}\ominus 0_{+}^{\mathcal{N}})\oplus (I_{-}^{\mathcal{N}})^{\perp}$, we only need to consider the form
$$
A=\left(\begin{array}{ccc}
A_{11}&0&A_{13}\\
0&A_{22}&A_{23}\\
0&0&0
\end{array}\right)\in\mathcal{A}_{1}.
$$
By Eqs. (\ref{14}) and (\ref{12}), the mapping $\epsilon_{1}$ satisfies
$$
A_{1}=\left(\begin{array}{ccc}
A_{11}&0&A_{13}\\
0&0&0\\
0&0&0
\end{array}\right)\mapsto
\left(\begin{array}{ccc}
\tilde{A}_{11}&\tilde{A}_{12}&\tilde{A}_{13}\\
0&0&0\\
0&0&0
\end{array}\right)
$$
and
$$
A_{2}=\left(\begin{array}{ccc}
0&0&0\\
0&A_{22}&A_{23}\\
0&0&0
\end{array}\right)\mapsto
\left(\begin{array}{ccc}
0&\ast&\ast\\
0&\ast&\ast\\
0&0&0
\end{array}\right).
$$
The first mapping implies that $(I_{-}^{\mathcal{N}}\ominus 0_{+}^{\mathcal{N}})(\phi_{1}(A_{1})+\tau(A_{1}))=0$, and consequently,
\begin{align*}
-\tau(A_{1})(I_{-}^{\mathcal{N}}\ominus 0_{+}^{\mathcal{N}})&=(I_{-}^{\mathcal{N}}\ominus 0_{+}^{\mathcal{N}})\phi_{1}(A_{1})\\
&=(I_{-}^{\mathcal{N}}\ominus 0_{+}^{\mathcal{N}})\phi_{1}(E_{0_{+}^{\mathcal{N}}}A_{1})\\
&=(I_{-}^{\mathcal{N}}\ominus 0_{+}^{\mathcal{N}})\phi_{1}(E_{0_{+}^{\mathcal{N}}})\phi_{1}(A_{1})\\
&=(I_{-}^{\mathcal{N}}\ominus 0_{+}^{\mathcal{N}})\epsilon_{1}(E_{0_{+}^{\mathcal{N}}})\phi_{1}(A_{1})\\
&=0,
\end{align*}
which implies $\tau_{1}(A_{1})=0$.

Similarly, we have $(\phi_{1}(A_{2})+\tau_{1}(A_{2}))0_{+}^{\mathcal{N}}=0$, and thus
$$
\phi_{1}(A_{2}) 0_{+}^{\mathcal{N}}=-\tau_{1}(A_{2}) 0_{+}^{\mathcal{N}}.
$$
Consequently,
\begin{align*}
-\tau_{1}(A_{2})0_{+}^{\mathcal{N}}&=\phi_{1}(A_{2})0_{+}^{\mathcal{N}}\\
&=\phi_{1}((E_{1}-E_{0_{+}^{\mathcal{N}}})A_{2})0_{+}^{\mathcal{N}}\\
&=\phi_{1}(E_{1}-E_{0_{+}^{\mathcal{N}}})\phi_{1}(A_{2})0_{+}^{\mathcal{N}}\\
&=\epsilon_{1}(E_{1}-E_{0_{+}^{\mathcal{N}}})\phi_{1}(A_{2})0_{+}^{\mathcal{N}}\\
&=0,
\end{align*}
so $\tau_{1}(A_{2})=0$. Therefore, $\tau_{1}=0$, and hence $\epsilon_{1}=\phi_{1}$ is an isomorphism.

Step 4: Let $\mathcal{N}_{2}=\{0,(I_{-}^{\mathcal{N}}\vee P_{\xi})\ominus I_{-}^{\mathcal{N}},I\ominus (I_{-}^{\mathcal{N}}\vee P_{\xi})\}$. For any
$$
B=\left(\begin{array}{cc}
0&eB_{2}\\
0&B_{2}
\end{array}\right)\in{\rm{Alg}\mathcal{L}},
$$
set
$$
E_{2}=\left(\begin{array}{cc}
0&e\\
0&1
\end{array}\right).
$$
By Eq. (\ref{13}), there exists a scalar $\lambda\in\mathbb{C}$ such that $\psi(B)=\tilde{B}+\lambda I$. Together with the Lie product $[E_{1},B]$, we may assume that
\begin{align}\label{4.1}
\tilde{B}=
\left(\begin{array}{cc}
0&\tilde{e}\tilde{B}_{2}\\
0&\tilde{B}_{2}
\end{array}\right).
\end{align}

Now consider the mapping
$$
\mathcal{T}(\mathcal{N}_{2})\to\mathcal{A}_{2}\overset{\epsilon_{2}}\to\tilde{\mathcal{A}}_{2}\to\mathcal{T}(\mathcal{N}_{2})
$$
where
$$
\left(\begin{array}{cc}
0&0\\
0&A_{2}\\
\end{array}\right)
\mapsto\left(\begin{array}{cc}
0&eA_{2}\\
0&A_{2}\\
\end{array}\right)\mapsto
\left(\begin{array}{cc}
0&\tilde{e}\tilde{A}_{2}\\
0&\tilde{A}_{2}\\
\end{array}\right)\mapsto
\left(\begin{array}{cc}
0&0\\
0&\tilde{A}_{2}
\end{array}\right).
$$

From Lie products $[E_{2},B]$ and $[C,E_{2}]$ for any
$$
C=\left(\begin{array}{cc}
0&\ast\\
0&0
\end{array}\right).
$$
Then we obtain
$$
\epsilon_{2}(E_{2})=\left(\begin{array}{cc}
0&\tilde{e}\\
0&1
\end{array}\right)(:=\tilde{E}_{2}).
$$

Let $B_{1}\in\mathcal{T}(\mathcal{N}_{1})$ and $B_{2}\in\mathcal{T}(\mathcal{N}_{2})$ be nontrivial idempotents and set
$$
X=\left(\begin{array}{cc}
B_{1}&-B_{1}e+eB_{2}\\
0&B_{2}
\end{array}\right).
$$
Then $X\in{\rm{Alg}\mathcal{L}}$ is an idempotent. By Eq. (\ref{3.1}) and Eq. (\ref{4.1}), there is a scalar $\lambda\in\mathbb{C}$ such $\psi(X)=\tilde{X}+\lambda I$, where
$$
\tilde{X}=\left(\begin{array}{cc}
\tilde{B}_{1}&-\tilde{B}_{1}\tilde{e}+\tilde{e}\tilde{B}_{2}\\
0&\tilde{B}_{2}
\end{array}\right).
$$
By Proposition \ref{idem} and the fact that $\tilde{B}_{1}$ is an idempotent, we conclude that $\tilde{B}_{2}$ is also an idempotent. Then, similarly to Step 3, it follows that $\epsilon_{2}$ coincides with an isomorphism on every idempotent in $\mathcal{A}_{2}$. By \cite{Pearcy-Toppinp}, every operator in $\mathcal{T}(\mathcal{N}_{2})$ can be expressed as a combination of at most three idempotents; the same holds for $\mathcal{A}_{2}$. Consequently, $\epsilon_{2}$ coincides with an isomorphism on every element of $\mathcal{A}_{2}$, and hence $\epsilon_{2}$ is itself an isomorphism.

Step 5: With respect to the decomposition $\mathcal{H}=I_{-}^{\mathcal{N}}\oplus (I_{-}^{\mathcal{N}})^{\perp}$, for any $A\in{\rm{Alg}}\mathcal{L}$, we have
$$
A=\left(\begin{array}{cc}
A_{11}&A_{12}\\
0&A_{22}\\
\end{array}\right),
$$
and we can express $A$ in ${\rm{Alg}}\mathcal{L}$ as $A=A_{1}+X+A_{2}$, where
$$
A_{1}=\left(\begin{array}{cc}
A_{11}&-A_{11}e\\
0&0\\
\end{array}\right),
A_{2}=\left(\begin{array}{cc}
0&eA_{22}\\
0&A_{22}\\
\end{array}\right)
$$
and
$$
X=\left(\begin{array}{cc}
0&A_{12}+A_{11}e-eA_{22}\\
0&0\\
\end{array}\right).
$$
Now, we can define the mapping $\epsilon$ from ${\rm{Alg}}\mathcal{L}$ onto ${\rm{Alg}}\mathcal{L}$ with
$$
\epsilon(A)=\epsilon_{1}(A_{1})+\phi(X)+\epsilon_{2}(A_{2}).
$$

Next, we prove that $\epsilon$ is an automorphism. Based on the preceding discussion, it remains only to establish the multiplicative property of $\epsilon$.

With respect to the decomposition $\mathcal{H}=I_{-}^{\mathcal{N}}\oplus (I_{-}^{\mathcal{N}})^{\perp}$, for any $A,B\in{\rm{Alg}}\mathcal{L}$, we have
$$
A=\left(\begin{array}{cc}
A_{11}&A_{12}\\
0&A_{22}\\
\end{array}\right) \ {\rm{and}} \
B=\left(\begin{array}{cc}
B_{11}&B_{12}\\
0&B_{22}\\
\end{array}\right).
$$
We decompose $A$ as $A=A_{1}+X+A_{2}$, and $B=B_{1}+Y+B_{2}$.

Clearly, $\epsilon$ preserves the Lie product; thus, by direct calculation, we obtain the following identities:

(1) Since $[A_{1},Y]=A_{1}Y$ and $[\epsilon(A_{1}),\epsilon(Y)]=\epsilon(A_{1})\epsilon(Y)$, it follows that
$$\epsilon(A_{1}Y)=\epsilon(A_{1})\epsilon(Y).$$

(2) For $A_{1}B_{2}=0$ and $\epsilon(A_{1})\epsilon(B_{2})=0$, we have
$$\epsilon(A_{1}B_{2})=\epsilon(A_{1})\epsilon(B_{2}).$$

(3) Since $[X,B_{1}]=-B_{1}X$ and $[\epsilon(X),\epsilon(B_{1})]=-\epsilon(B_{1})\epsilon(X)$, it follows that
$$\epsilon(B_{1}X)=\epsilon(B_{1})\epsilon(X).$$

(4) For $XY=0$ and $\epsilon(X)\epsilon(Y)=0$, we have
$$\epsilon(XY)=\epsilon(X)\epsilon(Y).$$

(5) Since $[X,B_{2}]=XB_{2}$ and $[\epsilon(X),\epsilon(B_{2})]=\epsilon(X)\epsilon(B_{2})$, it follows that
$$\epsilon(XB_{2})=\epsilon(X)\epsilon(B_{2}).$$

(6) For $A_{2}B_{1}=0$ and $\epsilon(A_{2})\epsilon(B_{1})=0$, we have
$$\epsilon(A_{2}B_{1})=\epsilon(A_{2})\epsilon(B_{1}).$$

(7) For $A_{2}Y=0$ and $\epsilon(A_{2})\epsilon(Y)=0$, we have
$$\epsilon(A_{2}Y)=\epsilon(A_{2})\epsilon(Y).$$

(8)  From the definition of $\epsilon$ and the fact that $\epsilon_{1}$ is an isomorphism, we obtain
$$
\epsilon(A_{1}B_{1})=\epsilon(A_{1})\epsilon(B_{1}).
$$

(9) A completely analogous argument yields
$$
\epsilon(A_{2}B_{2})=\epsilon(A_{2})\epsilon(B_{2}).
$$

In summary, we obtain the multiplicativity of $\epsilon$, and hence $\epsilon$ is an isomorphism.

Step 6: Let $\tau=\psi-\epsilon$. Then the conclusion of the theorem follows.
\end{proof}

\section{The case $I_{-}^{\mathcal{N}}<I$ with $I_{-}^{\mathcal{N}}\vee P_{\xi}=I$}
In this section we assume $I_{-}^{\mathcal{N}}<I$ with $I_{-}^{\mathcal{N}}\vee P_{\xi}=I$. Let $\mathcal{N}_{1}=\{N\in\mathcal{N}:N\leq I_{-}^{\mathcal{N}}\}$ and $\mathcal{T}(\mathcal{N}_{1})$ be the corresponding nest algebras. Then ${\rm{Alg}}\mathcal{L}$ is isomorphic to $\mathcal{T}(\mathcal{N}_{1})\oplus \mathbb{C}$. Now, we assume that $\psi$ is a Lie automorphism from $\mathcal{T}(\mathcal{N}_{1})\oplus \mathbb{C}$ onto itself.

\begin{theorem}
Let $\psi:\mathcal{T}(\mathcal{N}_{1})\oplus \mathbb{C}\to\mathcal{T}(\mathcal{N}_{1})\oplus \mathbb{C}$ be a Lie automorphism. Then there exist an (algebraic) isomorphism or a negative of an (algebraic) anti-isomorphism $\epsilon$ on $\mathcal{T}(\mathcal{N}_{1})$, a linear functional $\tau_{2}:\mathbb{C}\to\mathbb{C}$ and two linear functionals $\tau_{1},\tau$ vanishing on the commutators of $\mathcal{T}(\mathcal{N}_{1}),\mathcal{T}(\mathcal{N}_{1})\oplus \mathbb{C}$, respectively, such that
$$\psi=((\epsilon+\tau_{1})\oplus \tau_{2})+\tau.$$
\end{theorem}
\begin{proof}
For any
$$
A_{1}=\left(\begin{array}{cc}
A_{11}&0\\
0&0
\end{array}\right)\in \mathcal{T}(\mathcal{N}_{1})\oplus \mathbb{C},
$$
we have
$$
\psi(A_{1})=\left(\begin{array}{cc}
\bar{A}_{11}&0\\
0&\bar{A}_{22}
\end{array}\right)=
\left(\begin{array}{cc}
\tilde{A}_{11}&0\\
0&0
\end{array}\right)+\tau(A_{1})I,
$$
where $\tau:\mathcal{T}(\mathcal{N}_{1})\to \mathbb{C}$ is a linear functional. Define $\phi:\mathcal{T}(\mathcal{N}_{1})\to \mathcal{T}(\mathcal{N}_{1})$ by
$$
\phi(A_{1})=\psi(A_{1})-\tau(A_{1})I, \  A_{1}\in\mathcal{T}(\mathcal{N}_{1}).
$$
Then $\phi$ is a Lie automorphism on $\mathcal{T}(\mathcal{N}_{1})$. By the result in \cite{Marcoux-Sourour}, we have $\phi=\epsilon+\tau_{1}$, where $\epsilon$ is an (algebraic) isomorphism or a negative of an (algebraic) anti-isomorphism of $\mathcal{T}(\mathcal{N}_{1})$, and $\tau_{1}$ is a linear functional vanishing on each commutator of $\mathcal{T}(\mathcal{N}_{1})$.

Now consider
$$
A_{2}=\left(\begin{array}{cc}
0&0\\
0&a_{22}
\end{array}\right)\in \mathcal{T}(\mathcal{N}_{1})\oplus \mathbb{C}.
$$
Write
$$
\psi(A_{2})=\left(\begin{array}{cc}
\ast&0\\
0&\ast
\end{array}\right).
$$
We claim that $\psi(A_{2})$ has the form
$$
\psi(A_{2})=\left(\begin{array}{cc}
0&0\\
0&\tilde{a}_{22}
\end{array}\right)+\lambda I, \ \lambda\in\mathbb{C}.
$$
For any  maximal commutative Lie ideal $\mathcal{J}$ of $\mathcal{T}(\mathcal{N}_{1})\oplus \mathbb{C}$, there is a maximal commutative Lie ideal $\mathcal{J}_{0}$ of $\mathcal{T}(\mathcal{N}_{1})$ such that $\mathcal{J}=\mathcal{J}_{0}\oplus \mathbb{C}$. Since $A_{2}$ belongs to every maximal commutative Lie ideal, $\psi(A_{2})$ also belongs to every maximal commutative Lie ideal. Thus, if $0_{+}^{\mathcal{N}_{1}}=0$, the claim is clear.
If $0_{+}^{\mathcal{N}_{1}}\neq 0$, we use the decomposition $\mathcal{H}=0_{+}^{\mathcal{N}_{1}}\oplus (I_{-}^{\mathcal{N}}\ominus0_{+}^{\mathcal{N}_{1}})\oplus \mathbb{C}$, and we can assume that
$$
\psi(A_{2})=\left(\begin{array}{ccc}
0&\ast&0\\
0&0&0\\
0&0&\ast
\end{array}\right)+\lambda I, \ \lambda\in\mathbb{C}.
$$
Since $[A_{1},A_{2}]=0$ for any $A_{1}\in\mathcal{T}(\mathcal{N}_{1})$, we have $[\psi(A_{1}),\psi(A_{2})]=0$. This forces
$$
\psi(A_{2})=\left(\begin{array}{ccc}
0&0&0\\
0&0&0\\
0&0&\ast
\end{array}\right)+\lambda I, \ \lambda\in\mathbb{C}.
$$
Thus the claim is proved.

Consequently, there exists a linear functional $\tau_{2}:\mathbb{C}\to\mathbb{C}$ such that
$$
\psi(A_{2})=\left(\begin{array}{cc}
0&0\\
0&\tau_{2}(a_{22})
\end{array}\right)+\lambda I, \ \lambda\in\mathbb{C}.
$$

Now, we extend $\phi$ from $\mathcal{T}(\mathcal{N}_{1})$ to $\mathcal{T}(\mathcal{N}_{1})\oplus \mathbb{C}$ by setting $\psi(A_{2})-\phi(A_{2})\in\mathbb{C}I$; we still denote the extended map by $\phi$. Then $\phi=(\epsilon+\tau_{1})\oplus \tau_{2}$. Let $\tau=\psi-\phi$. Then $\tau_{1},\tau_{2},\tau$ satisfy the conditions stated in the theorem. This completes the proof.
\end{proof}

\noindent\textbf{Funding} The research of the first named author is supported by  NSF of China (12501161)  and NSF of Shandong Province,
China (ZR2025QC1514).
The third author is supported by
supported by NSF of China (12471124)  and NSF of Shandong Province,
China (ZR2025MS93).\\

\noindent\textbf{Data Availability} This article has no associated data. \\

\noindent\textbf{Declarations} \\

\noindent\textbf{Competing interests} The authors declared that they have no conflict of interest to this work.\\

\end{document}